\documentclass[12pt, notitlepage]{amsart}
\usepackage{latexsym, amsfonts, amsmath, amssymb, amsthm, cite, tabls, paralist}

\newtheorem{theorem}{Theorem}[section]
\newtheorem{lemma}[theorem]{Lemma}
\newtheorem{prop}[theorem]{Proposition}

\theoremstyle{remark}
\newtheorem{remark}[theorem]{Remark}

\theoremstyle{definition}
\newtheorem{defi}[theorem]{Definition}
\newtheorem*{acknowledgements}{Acknowledgements}

\theoremstyle{remark}

\newtheorem{case}{Case}

\numberwithin{equation}{section}

\newcommand{\N}{\mathbb{N}}
\newcommand{\Z}{\mathbb{Z}}

\newcommand{\R}{\mathbb{R}}

\newcommand{\C}{\mathbb{C}}

\newcommand{\1}{\mathbf{1}}

\usepackage{graphicx}
\usepackage{mathrsfs}

\usepackage{varioref}

\usepackage{hyperref}

\DeclareMathOperator{\measure}{meas}
\newcommand{\major}{\mathfrak{M}}

\newcommand{\CC}{\mathcal{C}}
\newcommand{\CE}{\mathcal{E}}
\newcommand{\CM}{\mathcal{M}}
\newcommand{\CN}{\mathcal{N}}
\newcommand{\CQ}{\mathcal{Q}}
\newcommand{\CR}{\mathcal{R}}
\newcommand{\CS}{\mathcal{S}}
\newcommand{\CY}{\mathcal{Y}}
\newcommand{\CT}{\mathcal{T}}
\newcommand{\FM}{\mathfrak{M}}
\newcommand{\FN}{\mathfrak{N}}
\newcommand{\FP}{\mathfrak{P}}
\newcommand{\FQ}{\mathfrak{Q}}
\newcommand{\Fm}{\mathfrak{m}}
\newcommand{\Fn}{\mathfrak{n}}
\newcommand{\Fp}{\mathfrak{p}}
\newcommand{\Fq}{\mathfrak{q}}

\newcommand{\e}{{\rm e}}
\newcommand{\ee}{\varepsilon}
\newcommand{\vphi}{\varphi}
\newcommand{\vth}{\vartheta}
\newcommand{\tchi}{{\tilde \chi}}
\newcommand{\cPhi}{{\check \Phi}}
\DeclareMathOperator{\card}{card}
\newcommand{\dd}{{\rm d}}

\newcommand{\ssum}[1]{\sum_{\substack{#1}}}
\newcommand{\summ}{\mathop{\sum\sum}}

\renewcommand{\mod}[1]{\ ({\rm mod\ }#1)}
\newcommand{\mods}[1]{\ ({\rm mod\ }#1)^\times}

\renewcommand{\Re}{\mathfrak{Re}}
\renewcommand{\Im}{\mathfrak{Im}}

\renewcommand{\setminus}{\smallsetminus}

\begin{document}

\title[Waring's problem with friable numbers]{Weyl sums, mean value estimates, \\ and Waring's problem with friable numbers}

\author{Sary Drappeau}
\address{Aix-Marseille Universit\'{e}, CNRS, Centrale Marseille \\ I2M UMR 7373 \\ 13453 Marseille Cedex \\ France}
\email{sary-aurelien.drappeau@univ-amu.fr}
\thanks{SD was supported by a CRM-ISM post-doctoral fellowship.}

\author{Xuancheng Shao}
\address{Mathematical Institute\\ Radcliffe Observatory Quarter\\ Woodstock Road\\ Oxford OX2 6GG \\ United Kingdom}
\email{Xuancheng.Shao@maths.ox.ac.uk}
\thanks{XS was supported by a Glasstone Research Fellowship.}

\maketitle

\begin{abstract}
In this paper we study Weyl sums over friable integers (more precisely~$y$-friable integers up to~$x$ when~$y = (\log x)^C$ for a large constant~$C$). In particular, we obtain an asymptotic formula for such Weyl sums in major arcs, nontrivial upper bounds for them in minor arcs, and moreover a mean value estimate for friable Weyl sums with exponent essentially the same as in the classical case. As an application, we study Waring's problem with friable numbers, with the number of summands essentially the same as in the classical case.
\end{abstract}

\section{Introduction}

\subsection{Waring's problem}

Posed in 1770, Waring's problem~\cite{Waring} is the question of whether or not, given a positive integer~$k$, there exist positive integers~$s$ and~$N_0$ such that every integer~$N>N_0$ can be written as a sum of~$s$ $k$-th powers:
\begin{equation}\label{eq:waring}
N = n_1^k + \cdots + n_s^k.
\end{equation}
Here and in the rest of the paper, by a~$k$-th power we mean the~$k$-th power of a non-negative integer. Call~$G(k)$ the least such number~$s$. After Hilbert~\cite{Hilbert} proved that~$G(k)<\infty$, there came the question of precisely determining the value of~$G(k)$. This question, usually attacked by the circle method, has motivated an outstanding amount of research in the theory of exponential sums. Referring to the survey by Vaughan and Wooley~\cite{VaughanWooley-survey} for a precise account of the vast history of this problem, we mention Wooley's state-of-the-art result~\cite{Woo95} that
\begin{equation}\label{eq:estim-Gk}
G(k) \leq k(\log k + \log\log k + 2 + O(\log\log k / \log k)).
\end{equation}
Conjecturally~$G(k) = O(k)$, and even~$G(k)=k+1$ if there are no ``local obstructions''.

To obtain an asymptotic formula for the number of solutions to the equation~\eqref{eq:waring}, we need more variables than the bound given in~\eqref{eq:estim-Gk}. The current best published result, following from Wooley's work~\cite{WooleyMGEC} on the Vinogradov main conjecture, gives such as asymptotic formula when
$$ s \geq Ck^2 + O(k) $$
for~$C=1.542749...$.  The Vinogradov main conjecture has very recently been proved by Bourgain, Demeter and Guth~\cite{BDG-Decoupling}, which would allow~$C=1$.

\subsection{Friable integers}

In this paper we study the representation problem \eqref{eq:waring} with the condition that the variables~$n_j$ have only small prime factors. Given~$y\geq 2$, a positive integer~$n$ is called \emph{$y$-friable}, or~\emph{$y$-smooth}, if its largest prime factor~$P(n)$ is at most~$y$. Estimates involving friable numbers have found applications in different areas in number theory. In fact they are a crucial ingredient in the proof of the estimate~\eqref{eq:estim-Gk} for~$G(k)$, and so are naturally studied in conjunction with Waring's problem. We refer to the surveys~\cite{HT-survey, Granville, Moree} for an account of classical results on friable numbers and their applications. 

The following standard notations will be used throughout the paper. For~$2 \leq y \leq x$, let
$$ S(x, y) := \{ n\leq x : P(n)\leq y\}, \qquad \Psi(x, y) := \card S(x, y). $$
The size of the parameter~$y$ with respect to~$x$ is of great importance in the study of friable numbers. The lower~$y$ is, the sparser the set~$S(x,y)$ is, and the more difficult the situation typically becomes. For example, when~$y = x^{1/u}$ for some fixed~$u \geq 1$, we have
\[ \Psi(x,y) \sim \rho(u) x \qquad (x\to\infty), \]
so that~$S(x,y)$ has positive density. Here~$\rho(u)$ is Dickman's function. On the other hand, when~$y = (\log x)^\kappa $ for some fixed~$\kappa>1$, we have
$$ \Psi(x, (\log x)^\kappa) = x^{1-1/\kappa + o(1)} \qquad (x\to\infty). $$
Because of this sparsity, many results about friable numbers from the second example above were until recently only known conditional on assumptions such as the Generalized Riemann Hypothesis.

The main result in our paper (Theorem~\ref{thm:waringk} below) is an asymptotic formula in Waring's problem with~$(\log N)^\kappa$-friable variables, when~$\kappa$ is sufficiently large. Here we state a special case of it.

\begin{theorem}\label{coro:principal-short}
For any given~$k\geq 2$, there exist~$\kappa(k)$ and~$s(k)$, such that every sufficiently large positive integer~$N$ can be represented in the form~\eqref{eq:waring}, with each~$n_j \in S(N^{1/k}, (\log N)^{\kappa})$. Moreover, we can take~$s(2) = 5$, $s(3) = 8$, and
$$ s(k) = k(\log k + \log\log k + 2 + O(\log\log k / \log k)) $$
for large~$k$.
\end{theorem}

An overview of the proof will be given in Section~\ref{sec:overview}. In the remainder of this introduction, we summarize some previous works on Waring's problem with friable variables.

\subsection{Past works}

If the variables are only required to be mildly friable (more precisely with the friability parameter~$\exp(c(\log N\log\log N)^{1/2})$ for some~$c > 0$ instead of~$(\log N)^{\kappa}$), then the existence of solutions to \eqref{eq:waring} with friable variables has been proved by Balog-S\'{a}rk\"{o}zy\cite{BalogSarkozy} (for~$k=1$), and Harcos~\cite{Har97} (for larger~$k$, using a key ingredient from~\cite{Woo92}). In the case~$k=3$, Br\"{u}dern and Wooley~\cite{BW01} proved that one can take~$s=8$ mildly friable variables.

The case~$k=2$ with~$4$ variables or less is particularly interesting, due to the failure of a naive application of the circle method. Without any restrictions on the variables, Kloosterman's refinement of the circle method can work (see~\cite[Chapter~20.3]{IK}), but there is no clear way to use it with friability restrictions. The best bound so far, achieved by Blomer, Br\"{u}dern, and Dietmann~\cite{BBD-smooth4} from Buchstab's identity to relax the friability condition, gets the allowable friable parameter~$y=x^{365/1184}$ . 

Finally, the most recent breakthough came in the case~$k=1$. This was first studied in the aforementioned work of Balog and S\'{a}rk\"{o}zy~\cite{BalogSarkozy} who obtain a lower bound for the number of solutions with~$s=3$ mildly friable variables. Assuming the Riemann hypothesis for Dirichlet~$L$-functions, Lagarias and Soundararajan~\cite{LagariasSound} improved the friability level to~$y=(\log N)^{8+\ee}$ for any~$\ee>0$. An asymptotic formula for the number of solutions was first reached in~\cite{Breteche-sommes}, using earlier results on friable exponential sums~\cite{FouvryTenenbaum, Breteche-expo}. Subsequent works~\cite{BG, D2014} eventually led to the friability level~$y= \exp\{c(\log N)^{1/2}(\log\log N)\}$ for some absolute~$c > 0$.

The situation changed drastically with the work of Harper~\cite{Harper} who proved \emph{unconditionally} that for~$k=1$, one can take~$s=3$ and~$y=(\log x)^C$ for large enough~$C$. This is the starting point of our present work; we show that Harper's approach can be adapted to treat higher powers as well, yielding results of comparable strength with what was previously known for mildly friable variables.

\begin{acknowledgements}
The authors are grateful to A. J. Harper, R. de la Bret{\` e}che and T. Wooley for helpful discussions and remarks. This work was started when XS was visiting the CRM (Montreal) during the thematic year in number theory theory in Fall 2014, whose hospitality and financial support are greatly appreciated.
\end{acknowledgements}

\section{Overview of results}\label{sec:overview}

In this section, we state the main result on Waring's problem with friable variables, as well as the exponential sum estimates required. 

To begin, we recall the ``saddle-point''~$\alpha(x, y)$ for~$2 \leq y \leq x$, introduced by Hildebrand and Tenenbaum~\cite{HT} and which is now standard in modern studies of friable numbers. It is defined by the implicit equation
\begin{equation}
\sum_{p\leq y} \frac{\log p}{p^\alpha - 1} = \log x.\label{eq:def-alpha}
\end{equation}
By~\cite[Theorem~2]{HT}, we have
\begin{equation}\label{eq:approx-alpha}
\alpha(x, y) \sim \frac{\log(1 + y/\log x)}{\log y}
\end{equation}
as~$y\to\infty$. In particular, for fixed~$\kappa\geq 1$, we have
$$ \alpha(x, (\log x)^\kappa) = 1 - 1/\kappa + o(1) \qquad (x\to\infty). $$
The relevance of~$\alpha$ to the distribution of friable numbers is hinted by the estimate~$\Psi(x, y) = x^{\alpha + o(1)}$ as~$x, y\to\infty$ (see de Bruijn~\cite{dB} and also~\cite[Theorem 1]{HT} for a more precise asymptotic of~$\Psi(x,y)$ in terms of the saddle point).

\subsection{Exponential sum estimates}

Throughout this paper, we use the standard notation
$$ \e(x) := \e^{2\pi i x} \qquad (x\in\C). $$
To study Waring's problem via the circle method, we need to understand the exponential sums
\[ E_k(x, y ; \vth) := \sum_{n\in S(x, y)}\e(n^k\vth) \qquad (\vth\in\R). \]
When~$\vth$ is approximated by a reduced fraction~$a/q$, we will frequently write
$$ \vth = \frac aq + \delta, \ \ \CQ = q(1 + |\delta| x^k), $$
where~$0 \leq a < q$ and~$(a,q) = 1$. Our estimate for~$E_k(x, y ; \vth)$ involves the ``local'' singular integral and singular series, defined by
\begin{equation}
\cPhi(\lambda, s) := s\int_0^1 \e(\lambda t^k)t^{s-1}\dd t \qquad (s\in\C, \Re(s)>0, \lambda\in\C),\label{eq:def-cPhi}
\end{equation}
\begin{equation}
H_{a/q}(s) := \ssum{d_1d_2|q \\ P(d_1d_2)\leq y} \frac{\mu(d_2)}{(d_1d_2)^s\vphi(q/d_1)}\ssum{b\mod{q}\\ (b, q)=d_1} \e\Big(\frac{ab^k}{q}\Big) \qquad (s\in\C).\label{eq:def-Hq}
\end{equation}

 In Section~\ref{sec:major} we prove the following major arc estimate, which generalizes~\cite[Th\'{e}or\`{e}me~4.2]{BG} and~\cite[Th\'{e}or\`{e}me~1.2]{D2014} to higher powers.

\begin{theorem}\label{thm:estim-majo}
Fix a positive integer~$k$. There exists~$C = C(k) > 0$ such that the following statement holds.  Let~$2 \leq y \leq x$ be large and let~$\alpha = \alpha(x,y)$. Let~$\vth \in [0,1]$ and write
$$ \vth = \frac aq + \delta, \ \ \CQ = q(1 + |\delta| x^k), $$
for some~$0 \leq a < q$ with~$(a,q)=1$. For any~$A, \ee > 0$, if~$y \geq (\log x)^{CA}$ and~$\CQ \leq (\log x)^A$, then
\begin{equation}
\begin{aligned}
\frac{E_k(x, y ; \vth)}{\Psi(x, y)} = &\ \cPhi(\delta x^k, \alpha)H_{a/q}(\alpha) + O_{\ee, A}\Big( \CQ^{-1/k+2(1-\alpha)+\ee} u_y^{-1} \Big),
\end{aligned}\label{eq:estim-majo}
\end{equation}
where~$u_y$ is defined in~\eqref{eq:notations-uy}. In particular, under the same conditions we have
\begin{equation}
E_k(x, y ; \vth) \ll_{\ee, A} \Psi(x, y) \CQ^{-1/k+2(1-\alpha)+\ee}.
\label{eq:arcmaj-majoration}
\end{equation}
\end{theorem}

Here~$u = (\log x)/\log y$ as usual. By \eqref{eq:approx-alpha}, we can make~$1-\alpha$ in the statement above arbitrarily small by taking~$A$ large enough. Thus the upper bound~\eqref{eq:arcmaj-majoration} has nearly the same strength as the classical major arc estimates for complete exponential sums.

In Section~\ref{sec:minor} we prove the following minor arc bound, which involves generalizing \cite[Theorem 1]{Harper} to higher powers.

\begin{theorem}\label{thm:estim-minor}
Fix a positive integer~$k$.  There exists~$K = K(k) > 0$ and~$c = c(k) > 0$ such that the following statement holds. Let~$2 \leq y \leq x$ be large with~$y \geq (\log x)^K$. Assume that~$|\vth - a/q| \leq 1/q^2$ for some~$0 \leq a < q$ with~$(a, q)=1$. Then
$$ E_k(x, y ; \vth) \ll \Psi(x,y) \Big(\frac1q + \frac{q}{x^k}\Big)^c. $$
\end{theorem}

For mildly friable variables, this was proved by Wooley \cite[Theorem 4.2]{Woo95b}, with a very good exponent~$c(k) \asymp (k \log k)^{-1}$. By following the proof, one can prove Theorem~\ref{thm:estim-minor} with~$c(k)$ depending on~$k^{-1}$ polynomially. 

\subsection{Mean value estimates}

We complement the estimates of the previous sections by the study of moments:
\begin{equation}
\int_0^1|E_k(x, y ; \vth)|^p \dd\vth, \qquad (p\geq 0).\label{eq:def-moment}
\end{equation}
Indeed, the exponential sum estimates described above lead to Corollary~\ref{coro:principal-short} for some (potentially large)~$s$. To reduce the number of variables, we need the following mean value estimate, which generalizes \cite[Theorem 2]{Harper} to higher powers.  We refer to the introduction of~\cite{Harper} for a detailed explanation on the necessity of such a mean value estimate when dealing with a sparse set of friable numbers.

\begin{theorem}\label{cor:mv-k}
Fix a positive integer~$k$.  Let~$2 \leq y \leq x$ be large and let~$\alpha = \alpha(x,y)$. There exists~$p_0 = p_0(k) \geq 2k$ such that for any~$p>p_0$, we have
\[
\int_0^1 | E_k(x, y; \vth) |^{p} \dd\vth \ll_{p,k} \Psi(x,y)^p x^{-k}, 
\]
provided that~$1-\alpha \leq c \min(1, p-p_0)$ for some sufficiently small~$c = c(k) > 0$. Moreover, we may take~$p_0(1)=2$, $p_0(2)=4$, and $p_0(3) = 8$. If~$y \leq x^c$ for some sufficiently small~$c = c(k) > 0$, then we may take~$p_0(3)=7.5907$ and~$p_0(k)=k(\log k + \log\log k + 2 + O(\log\log k / \log k)$ for large~$k$.
\end{theorem}

Conjecturally, the choice~$p_0(k)=2k$ should be admissible. The admissible choices of~$p_0(k)$ for~$k=3$ and for large~$k$ in the statement above are essentially the same as the best known thresholds for the corresponding problem with mildly friable numbers. This ultimately allows us to prove Corollary~\ref{coro:principal-short} with essentially the same number of variables as in previous works for mildly friable numbers.

\subsection{Application to Waring's problem}

For readers familiar with the circle method, it is a rather routine matter to deduce from the estimates above the following theorem, of which Corollary~\ref{coro:principal-short} is an immediate consequence. This deduction will be carried out in Section~\ref{sec:waring}.

\begin{theorem}\label{thm:waringk}
Fix a positive integer~$k$. There exists~$s_0 = s_0(k)$ such that the following statement holds for all positive integers~$s \geq s_0$. Let~$N$ be a large positive integer, let~$x = N^{1/k}$, and let~$2 \leq y \leq x$. Then the number of ways to write 
\[ N=n_1^k+\cdots+n_s^k \]
with each~$n_j\in S(x,y)$ is
\[ x^{-k} \Psi(x,y)^s \left( \beta_{\infty} \prod_p \beta_p + O_s(u_y^{-1}) \right), \]
where~$u_y$ is defined in~\eqref{eq:notations-uy}, provided that~$y \geq (\log x)^{C}$ for some sufficiently large~$C = C(k) > 0$. Here the archimedean factor~$\beta_{\infty}$ and the local factors~$\beta_p$ are defined in~\eqref{eq:arch-factor} and~\eqref{eq:local-factor} below, respectively.
Moreover, we may take~$s_0(1)=3$, $s_0(2)=5$, and $s_0(3) = 9$. If~$y \leq x^c$ for some sufficiently small~$c = c(k) > 0$, then we may take~$s_0(3)=8$ and~$s_0(k)=k(\log k + \log\log k + 2 + O(\log\log k / \log k))$ for large~$k$.
\end{theorem}

By Propositions~\ref{prop:beta-infty} and~\ref{prop:beta-p}, both~$\beta_{\infty}$ and~$\prod_p \beta_p$ are positive with the given choices of~$s_0(k)$ and the assumption on~$y$. Thus Corollary~\ref{coro:principal-short} indeed follows.

Our technique (in particular Proposition~\ref{thm:restriction} below), combined with estimates in~\cite{Woo95b}, allows to show that every large positive integer is the sum of six friable cubes and one unrestrained cube. In the mildly friable case, this was observed by Kawada~\cite{Kaw05}. We will not give the details here.

\subsection*{Notations}

We use the following standard notations. For~$2 \leq y \leq x$, we write
\begin{equation}\label{eq:notations-uy} 
u := (\log x)/\log y, \quad \frac1{u_y} := \min\Big\{\frac1u, \frac{\log(1+u)}{\log y}\Big\},  \quad H(u) := \exp\Big\{\frac{u}{(\log(u+1))^2}\Big\}.
\end{equation}
We will also denote
\begin{equation}
Y := \min\{y, \e^{\sqrt{\log x}}\}, \qquad \CY_\ee := \e^{(\log y)^{3/5-\ee}}, \qquad \CT_\ee := \min\{\e^{(\log y)^{3/2-\ee}}, H(u)\}.\label{eq:notations}
\end{equation}
Throughout we fix a positive integer~$k$, and all implied constants are allowed to depend on~$k$. We will always write~$\alpha = \alpha(x,y)$, and will frequently assume that~$1 - \alpha$ is sufficiently small, or equivalently~$y \geq (\log x)^C$ for some sufficiently large~$C$.

\section{Lemmata}

\subsection{Friable numbers}

We recall the definition~\eqref{eq:def-alpha} of the saddle-point~$\alpha(x, y)$. It is the positive real saddle point of the associated Mellin transform~$x^s \zeta(s, y)$, where
$$ \zeta(s, y) := \sum_{P(n)\leq y} n^{-s} = \prod_{p\leq y}(1-p^{-s})^{-1} \qquad (\Re s > 0). $$
Let
$$ \sigma_2(\alpha, y) := - \frac{\dd}{\dd\alpha} \sum_{p \leq y} \frac{\log p}{p^{\alpha}-1} =  \sum_{p\leq y}\frac{(\log p)^2p^\alpha}{(p^\alpha-1)^2}. $$
Then from Hildebrand--Tenenbaum~\cite{HT}, we have the uniform estimate
\begin{equation}\label{eq:HT-psixy} 
\Psi(x, y) = \frac{x^\alpha \zeta(\alpha, y)}{\alpha\sqrt{2\pi \sigma_2(\alpha, y)}}\Big\{1 + O\Big(\frac1u + \frac{\log y}y\Big)\Big\} \qquad (2\leq y\leq x). 
\end{equation}
Note that for~$y \gg \log x$ we have
\begin{equation}\label{eq:sigma2} 
\sigma_2(\alpha, y) \asymp (\log x)\log y.
\end{equation} 
The saddle-point~$\alpha$ belongs to the interval~$(0, 1)$ for large enough~$x$ (independently of~$y$ with~$2\leq y\leq x$). We have
\begin{equation}
1-\alpha = \frac{\log (u \log (u+1))}{\log y} + O\left( \frac{1}{\log y} \right) \qquad (\log x \leq y \leq x). \label{eq:alpha-1}
\end{equation}

\subsection{Friable character sums}
\label{sec:lemmes-friables-carac}

In this section, we regroup facts about the character sums
$$ \Psi(x, y ; \chi) := \sum_{n\in S(x, y)} \chi(n), $$
where~$\chi$ is a Dirichlet character. We quote the best known results from work of Harper~\cite{HarperBV}. For some absolute constants~$K, c>0$, with~$K$ large and~$c$ small, the following is true. Assume that
$$ 3\leq (\log x)^K \leq y \leq x .$$

We recall the notations~\eqref{eq:notations}. Proposition~3 of \cite{HarperBV} implies that the bound
\begin{equation}
\Psi(x, y ; \chi) \ll \Psi(x, y) Y^{-c}\label{eq:borne-carac-normal}
\end{equation}
holds for any Dirichlet character~$\chi$ of modulus less than~$x$, of conductor less than~$Y^c$, and whose Dirichlet~$L$-function has no zero in the interval~$[1-K/\log Y, 1]$.

Secondly, among all primitive Dirichlet character~$\chi$ of conductor at most~$Y^c$, there is at most one which does not satisfy the above bound. If such a character~$\chi_1$ exists and has conductor~$q_1$, say, then any character~$\chi$ induced by~$\chi_1$ and of modulus~$q\leq x$ satisfies
\begin{equation}
\Psi(x, y ; \chi) \ll \Psi(x, y)\frac{\log q_1}{\log x}\Big(\sum_{d|(q/q_1)}d^{-\alpha}\Big) \big\{y^{-c} + H(u)^{-c}\big\}.\label{eq:borne-carac-siegel}
\end{equation}
This is deduced from the computations in~\cite[\S{}3]{HarperBV} (see in particular the first formula on page 16, and the last formula on page 17).

\subsection{Higher order Gauss sums}

Important for our study will be the following generalisation of Gauss sums. Given the integers~$k\geq 1$,~$q\geq 1$, a residue class~$a\mod{q}$ and a character~$\chi\mod{q}$, we let
\[ G_k(q, a, \chi) := \sum_{b\mods{q}}\chi(b)\e\Big(\frac{ab^k}q\Big). \]
We have the following bound.
\begin{lemma}\label{lemme:gauss}
Suppose~$q, a, a'$ are positive integers, and~$\chi$ is a character modulo~$q$. Suppose~$(a', q)=1$, and let~$q^*|q$ denote the conductor of~$\chi$. Then
\[ |G_k(q, aa', \chi)| \leq 2k^{\omega(q)}\tau(q)\min\{q/\sqrt{q^*}, \sqrt{aq}\}. \]
\end{lemma}
\begin{proof}
By using orthogonality of additive and multiplicative characters modulo~$q$, it is easily seen that
\begin{align*}
\sum_{b\mods{q}}\chi(b)\e\Big(\frac{aa'b^k}{q}\Big) = \ssum{\tchi\mod{q}\\\tchi^k\chi=\chi_0}\sum_{c\mods{q}}\e\Big(\frac{aa'c}q\Big)\overline{\tchi(c)}.
\end{align*}
For each~$\tchi$ in the above, the inner sum over~$c$ is a Gauss sum, so that by \textit{e.g.}~\cite[Lemma~3.2]{IK},
\[ \Bigg|\sum_{c\mods{q}}\e\Big(\frac{aa'c}q\Big)\overline{\tchi(c)}\Bigg| \leq \sqrt{q'}\sum_{d|(q/q', a)}d \]
where~$q'|q$ is the conductor of~$\tchi$. Here we used our assumption that~$(a', q)=1$. The fact that~$\tchi^k\chi=\chi_0$ imposes that~$q^*|q'$. Writing~$q'=rq^*$, we have~$r|q/q^*$ and so
\[ |G_k(q, a a', \chi)| \leq \big|\{\tchi\mod{q} : \ \tchi^k\chi = \chi_0\}\big| \big(\sup_{r|q/q^*}\sqrt{rq^*}\sum_{d|(q/(rq^*), a)}d\big). \]
The sum over~$d$ has at most~$\tau(q)$ terms, and so we trivially have
\[ \sup_{r|q/q^*}\sqrt{rq^*}\sum_{d|(q/(rq^*), a)}d \leq \tau(q)\sqrt{q^*}\sup_{r\in[1, q/q^*]}\min\{q/(q^*\sqrt{r}), a\sqrt{r}\}. \]
The supremum over~$r$ evaluates to~$\min\{q/q^*, \sqrt{aq/q^*}\}$. Therefore,
\[ \sup_{r|q/q^*} \sqrt{rq^*}\sum_{d|(q/(rq^*), a)}d \leq \tau(q)\min\{q/\sqrt{q^*}, \sqrt{aq}\}. \]
To conclude it suffices to show that there are at most~$2k^{\omega(q)}$ characters~$\tchi$ satisfying~$\tchi^k\chi = \chi_0$. By the Chinese remainder theorem, the group of characters of~$(\Z/q\Z)^\times$ is isomorphic to a product of~$\omega(q)$ cyclic groups (where~$\omega(q)$ is the number of distinct prime factors of~$q$), and possibly~$\{\pm1\}$. Therefore, the number of characters~$\tchi\mod{q}$ satisfying~$\tchi^k\chi=\chi_0$ is at most~$2k^{\omega(q)}$. This yields our lemma.
\end{proof}

\subsection{Friable numbers in short intervals}

We will need the following two upper bounds on the number of~$y$-friable numbers in short intervals. These upper bounds are almost sharp for a very wide range of~$y$ and the length of the short intervals.

\begin{lemma}\label{lem:adam-smooth-2}
For any~$2 \leq y \leq x$ and~$d \geq 1$, we have
\[ \Psi(x/d,y) \ll d^{-\alpha(x, y)} \Psi(x,y). \]
\end{lemma}

\begin{proof}
See~\cite[Theorem 2.4]{BT2005}.
\end{proof}

\begin{lemma}\label{lem:adam-smooth-3}
Let~$\log x \leq y \leq x$ be large. For any arithmetic progression~$I \subset [x,2x] \cap \Z$, we have
\[ | \{n \in I: P^+(n) \leq y \} | \ll |I|^{\alpha} \frac{\Psi(x,y)}{x^{\alpha}} \log x, \]
where~$\alpha = \alpha(x,y)$.
\end{lemma}

\begin{proof}
When~$|I| \geq y$, this is Smooth Number Result 3 in~\cite[Section 2.1]{Harper}. When~$|I| \leq y$, we can bound the left side trivially by~$|I|$ and the right side is~$\gg |I|^{\alpha} \log x \gg |I|$ by~\eqref{eq:alpha-1}.
\end{proof}

\subsection{Equidistribution results}

In our proof of the mean value estimates, we will need the following equidistribution-type results. The first is the classical Erd\"{o}s-Tur\'{a}n inequality, connecting equidistribution of points with exponential sums.

\begin{lemma}[Erd\"{o}s-Tur\'{a}n]\label{lem:erdos-turan}
Let~$\vth_1, \ldots, \vth_N \in \R/\Z$ be arbitrary. Then for any interval~$I \subset \R/\Z$ and any positive integer~$J$, we have
\[ \left| \#\{1 \leq n \leq N: \vth_n \in I\} - N \cdot \measure(I) \right| \leq \frac{N}{J+1} + 3 \sum_{j=1}^J \frac{1}{j} \left| \sum_{n=1}^N \e(j \vth_n) \right|. \]
\end{lemma}

\begin{proof}
See~\cite[Corollary 1.1]{ten-lectures}.
\end{proof}

We also need the following result about well spaced points in major arcs, used in the restriction argument of Bourgain~\cite{Bou89} (see also~\cite[Section 2.2]{Harper}). 

\begin{lemma}\label{lem:restriction}
Let~$x$ be large. Let~$Q \geq 1$ and~$1/x \leq \Delta \leq 1/2$ be parameters. For~$\vth \in \R$ define
\[ G_{x,Q,\Delta}(\vth) = \sum_{q \leq Q} \frac{1}{q} \sum_{a=0}^{q-1} \frac{ \1_{\|\vth - a/q\| \leq \Delta} }{1 + x\|\vth - a/q\|}. \]
For any~$\vth_1,\cdots,\vth_R \in \R$ satisfying the spacing conditions~$\| \vth_r - \vth_s\| \geq 1/x$ whenever~$r \neq s$, we have
\[ \sum_{1 \leq r,s \leq R} G_{x,Q,\Delta}(\vth_r - \vth_s) \ll_{\ee, A} \left( RQ^{\varepsilon} + \frac{R^2Q}{x} + \frac{R^2}{Q^A} \right) \log (1+\Delta x), \]
for any~$\varepsilon, A > 0$.
\end{lemma}

When we apply this, the first term on the right will dominate, showing that the main contribution to the sum on the left comes from the diagonal terms with~$r = s$.

\subsection{Variants of the Vinogradov lemma}

We also need the following variants of the Vinogradov lemma, which concerns diophantine properties of strongly recurrent polynomials. The proof of the following lemma can be found in~\cite[Lemma~4.5]{GT-manifold}.

\begin{lemma}\label{lem:vinogradov-polynomial}
Let~$k$ be a fixed positive integer and let~$\varepsilon, \delta \in (0,1/2)$ be real. Suppose that, for some~$\vth \in \R$, there are at least~$\delta M$ elements of~$m \in [-M,M] \cap \Z$ satisfying~$\| m^k \vth \| \leq \varepsilon$. If~$\varepsilon < \delta/5$, then there is a positive integer~$q \ll \delta^{-O(1)}$ such that~$\| q \vth \| \ll \delta^{-O(1)}\varepsilon / M^k$.
\end{lemma}

The next lemma allows us to deal with cases where diophantine information is only available in a sparse set~$A$, which will taken to be the set of friable numbers in our application.

\begin{lemma}\label{lem:vinogradov}
Let~$k$ be a fixed positive integer and let~$\varepsilon, \delta \in (0,1/2)$ be real. Let~$1 \leq L \leq M$ be positive integers and let~$A \subset [M,2M]\cap\Z$ be a non-empty subset satisfying
\[ |A \cap P| \leq \Delta \frac{|A||P|}{M} \]
for any arithmetic progression~$P \subset [M,2M] \cap \Z$ of length at least~$L$ and some~$\Delta \geq 1$. Suppose that, for some~$\vth \in \R$ with~$\|\vth\| \leq \varepsilon/(LM^{k-1})$, there are at least~$\delta |A|$ elements of~$m \in A$ satisfying~$\|m^k\vth\| \leq \varepsilon$. Then either~$\varepsilon \gg \delta/\Delta$ or~$\|\vth\| \ll \Delta \delta^{-1} \varepsilon / M^k$.
\end{lemma}

If the host set~$A$ is equidistributed, we can expect to take~$\Delta \asymp 1$, and thus the lemma upgrades the diophantine property of~$\vth$ significantly (if~$M$ is much larger than~$L$) under the strong recurrence of~$m^k\vth$.

\begin{proof}
We may assume that~$\varepsilon < 4^{-k}$ and~$\vth \neq 0$, since otherwise the conclusion holds trivially. We may also assume~$\vth\in[-1/2,1/2]$, so that~$\|\vth\|=|\vth|$. Let~$L' = \min(1/(4^k M^{k-1} |\vth|), M)$ be a parameter, and note that~$L' \geq \min(L/(4^k\varepsilon), M) \geq L$ by our assumption on~$\vth$. Let~$P' \subset [M,2M] \cap\Z$ be any interval of length~$L'$,  and take two arbitrary elements~$m_1,m_2 \in A \cap P'$ with~$\|m_1^k\vth\|, \|m_2^k\vth\| \leq \varepsilon$. Note that
\[ | m_1^k\vth - m_2^k\vth | \leq k (2M)^{k-1} |(m_1-m_2)\vth| \leq k (2M)^{k-1}L' |\vth|  < 1/2  \]
by our choice of~$L'$. Thus from the inequality
\[ \| m_1^k\vth - m_2^k\vth \| \leq \|m_1^k\vth\| + \|m_2^k\vth\| \leq 2\varepsilon \]
we deduce that~$|m_1^k\vth - m_2^k\vth| \leq 2\varepsilon$, and thus 
\[ |m_1 - m_2| \ll \frac{\varepsilon}{M^{k-1} |\vth|}. \]
We have just shown that all the integers~$m \in A \cap P'$ with~$\|m^k \vth\| \leq \varepsilon$ must lie in an interval of length~$O( \varepsilon/(M^{k-1}|\vth|) )$. Since~$\varepsilon / (M^{k-1}|\vth|) \geq L$ by the assumption on~$|\vth|$, our hypothesis implies that the number of integers~$m \in A \cap P'$ with~$\|m^k \vth\| \leq \varepsilon$  is 
\[ O\left( \frac{\Delta |A|}{M} \cdot \frac{\varepsilon}{M^{k-1} |\vth|} \right) = O\left( \frac{\Delta\varepsilon |A|}{M^k |\vth|} \right). \]
By covering~$[M,2M] \cap \Z$ by~$O(M/L')$ intervals of length~$L'$ and recalling the choice of~$L'$, we obtain
\[ \sum_{m \in A} 1_{\|m^k\vth\| \leq \varepsilon} \ll  \frac{\Delta\varepsilon |A|}{M^k |\vth|}  \cdot \frac{M}{L'} \ll \frac{\Delta \varepsilon |A|}{M^{k-1} |\vth|} \left( M^{k-1}|\vth| + \frac{1}{M} \right) = \Delta\varepsilon |A| + \frac{\Delta\varepsilon |A|}{M^k |\vth|}. \]
The left side above is at least~$\delta |A|$ by hypothesis, and thus
\[ \max\left( \Delta\varepsilon, \frac{\Delta\varepsilon}{M^k|\vth|} \right) \gg \delta. \]
This immediately leads to the desired conclusion.
\end{proof}

\section{Major arc estimates}\label{sec:major}

The goal of this section is to prove Theorem~\ref{thm:estim-majo}. We recall that the local factors~$\cPhi(\lambda, s)$ and~$H_{a/q}(s)$ were defined in~\eqref{eq:def-cPhi} and~\eqref{eq:def-Hq}, respectively. The following lemmas give bounds for~$\cPhi(\lambda, s)$ and~$H_{a/q}(s)$.

\begin{lemma}\label{lem:phi-check}
Fix a positive integer~$k$. For all~$\lambda, s \in \C$ with~$\sigma=\Re(s)\in(0, 1]$ and~$\Im(s)\ll 1$, and all~$j\geq 0$, we have
\[ \frac{\partial^j\cPhi}{\partial s^j}(\lambda, s) \ll_{j} \frac{(\log(2+|\lambda|))^{j} + \sigma^{-j}}{1+|\lambda|^{\sigma/k}}. \]
\end{lemma}
\begin{proof}
This follows from~\cite[Lemma~2.4]{D2014}, by a change of variables~$t\gets t^{1/k}$.
\end{proof}

\begin{lemma}\label{lem:Haq-bound}
Fix a positive integer~$k$. For all~$0 \leq a < q$ with~$(a,q) = 1$, and all~$\alpha \in (0,1]$, we have
\[ H_{a/q}(\alpha) \ll_\ee q^{-\alpha/k+\ee} \]
for any~$\ee > 0$.
\end{lemma}

\begin{proof}
This follows from Lemmas~\ref{lem:S=H(a/q)} and~\ref{lem:Sqa} in the appendix.
\end{proof}

The plan of this section is the following. A standard manipulation decomposes the exponential phase~$\e(n^k\vth)$ into a periodic part~$\e(n^ka/q)$, and a perturbation~$\e(n^k\delta)$. In Section~\ref{sec:non-princ-char}, we handle the twist by~$\e(n^ka/q)$ using results about friable character sums. In Sections~\ref{sec:main-term-0-big-y} and~\ref{sec:main-term-0-small-y}, we evaluate the exponential sum around~$\vth=0$, using the asymptotic formula for~$\Psi(x,y)$ and partial summation for large~$y$, and the saddle point method for small~$y$. In Section~\ref{sec:main-term-general}, we extend the analysis to all of the major arcs, using ``semi-asymptotic'' results about~$\Psi(x, y)$.

\subsection{Handling the non-principal characters}\label{sec:non-princ-char}

For~$\vth = a/q + \delta$ with~$0 \leq a < q$ and~$(a,q)=1$, we define the contribution of the principal characters to be
\begin{equation}
M_k(x, y ; \vth) = \ssum{d_1d_2|q\\P(d_1d_2)\leq y} \frac{\mu(d_2)}{\vphi(q/d_1)}\ssum{b\mod{q}\\(b, q)=d_1}\e\Big(\frac{ab^k}{q}\Big) E_k\big(\frac x{d_1d_2}, y ; (d_1d_2)^k\delta\big). \label{Mk}
\end{equation}
The exact form of this contribution will be clear from the first few lines of the proof of Proposition~\ref{prop:Ek-Mk} below, which says that the contributions from non-principal characters are negligible. Recall the notations from \eqref{eq:notations}.

\begin{prop}\label{prop:Ek-Mk}
There exist~$K, c>0$ such that under the condition
\begin{equation}\label{eq:arcmaj-1}
(\log x)^K\leq y\leq x, \qquad q(1+|\delta|x^k) \leq Y^c,
\end{equation}
we have
\begin{equation}
E_k(x, y ; \vth) = M_k(x, y ; \vth) + O_A(\Psi(x, y)(1+|\delta|x^k)(y^{-c} + H(u)^{-c}(\log x)^{-A}))\label{eq:Ek-Mk}
\end{equation}
for any~$A > 0$.
\end{prop}
\begin{proof}
Consider first the case when~$\delta=0$ (so that~$\vth = a/q$). We decompose
\begin{align*}
E_k(x, y ; a/q) =\ &\sum_{b\mod{q}}\e\Big(\frac{b^k a}q\Big)\ssum{n\in S(x, y) \\ n \equiv b\mod{q}} 1 \\
=\ &\ssum{d|q\\P(d)\leq y}\ssum{b\mod{q}\\(b,q)=d}\e\Big(\frac{ab^k}{q}\Big)\ssum{n\in S(x/d, y)\\n\equiv b/d \mod{q/d}}1 \\
=\ &\ssum{d|q\\P(d)\leq y}\frac1{\vphi(q/d)}\sum_{\chi\mod{q/d}}G_k(q/d, ad^{k-1}, \overline{\chi})\Psi(x/d, y ; \chi).
\end{align*}
The contribution of the principal character~$\chi = \chi_0$ is precisely~$M_k(x, y ; a/q)$ since
\[ \Psi(x/d,y; \chi_0) = \ssum{n \in S(x/d,y) \\ (n,q/d)=1} 1 = \sum_{d_2 \mid q/d} \mu(d_2) E_k( x/(d d_2), y; 0). \]
For the non-principal characters, we apply the bounds~\eqref{eq:borne-carac-normal} and~\eqref{eq:borne-carac-siegel}. We split the non-principal characters into two categories, according to whether or not the associated Dirichlet series has a real zero in the interval~$[1-K/\log Y, 1]$. Define a character to be \emph{normal} if its Dirichlet series has no such zero, and \emph{exceptional} if it does. The exceptional characters, if exist, consist of characters induced by a unique real primitive character~$\chi_1$ of conductor~$q_1$, say. Let
\[ \CN := \ssum{d|q\\P(d)\leq y}\frac1{\vphi(q/d)}\ssum{\chi\mod{q/d}\\\chi\text{ is normal}}G_k(q/d, ad^{k-1}, \overline{\chi})\Psi(x/d, y ; \chi), \]
\[ \CE := \ssum{d|q/q_1\\P(d)\leq y}\frac1{\vphi(q/d)}\ssum{\chi\mod{q/d}\\\chi\text{ is exceptional}}G_k(q/d, ad^{k-1}, \overline{\chi})\Psi(x/d, y ; \chi). \]
To bound~$\CN$, we use the trivial bound
\begin{equation}
|G_k(q/d, ad^{k-1}, \overline{\chi})| \leq q/d,\label{eq:borne-Gk-triviale}
\end{equation}
and Lemma~\ref{lem:adam-smooth-2}. Note that~$\log(x/q) \asymp \log x$, so that uniformly over~$d\leq q$ and all normal characters~$\chi$, we have
\[ \Psi(x/d, y ; \chi) \ll d^{-\alpha}\Psi(x, y)Y^{-c} .\]
Combining this with the trivial bound~\eqref{eq:borne-Gk-triviale}, we obtain
\begin{equation}
\begin{aligned}
\CN \ll \Psi(x, y)Y^{-c}q\sum_{d|q} d^{-1-\alpha} \ll \Psi(x, y)Y^{-c/2},
\end{aligned}\label{eq:majo-cN}
\end{equation}
given our hypothesis~\eqref{eq:arcmaj-1}.

We now bound~$\CE$. The upper bound we have for the character sum~$\Psi(x, y ; \chi)$ is very poor when~$u$ is small, therefore, more care must be taken. We have by Lemma~\ref{lemme:gauss}
\begin{equation} 
|G_k(q/d, ad^{k-1}, \overline{\chi})| \leq 2k^{\omega(q)}\tau(q)\min\{q/(d\sqrt{q_1}), \sqrt{d^{k-2}q}\} \ll_\ee q^{\ee}\sqrt{q_1}\Big(\frac{q}{q_1}\Big)^{1-1/k}.\label{eq:borne-Gk-nontv}
\end{equation} 
Thus
\begin{equation}
\CE \ll q^\ee\sqrt{q_1}\Big(\frac{q}{q_1}\Big)^{1-1/k}\ssum{d|q/q_1} \frac{|\Psi(x/d, y ; \chi_{q/d})|}{q/d}\label{eq:borne-cE}
\end{equation}
where~$\chi_{q/d}$ stands for the character~$\mod{q/d}$ induced by~$\chi_1$. For the same reason as before, since~$\log(x/d)\asymp \log x$, the character sum bound~\eqref{eq:borne-carac-siegel} can be applied with~$x$ replaced by~$x/d$ and yields
\[ |\Psi(x/d, y ; \chi_{q/d})| \ll_\ee d^{-\alpha} q^\ee \Psi(x, y)(H(u)^{-c}+y^{-c}). \]
We deduce
\begin{align*}
\CE &\ll \Psi(x, y)(H(u)^{-c}+y^{-c})q^\ee\frac{\sqrt{q_1}}q\Big(\frac{q}{q_1}\Big)^{1-1/k}\sum_{d|q/q_1} d^{1-\alpha} \\ &\ll \Psi(x, y)(H(u)^{-c}+y^{-c})\frac{q^\ee}{\sqrt{q_1}}\Big(\frac{q}{q_1}\Big)^{1-\alpha-1/k}.
\end{align*}
Assuming that~$K$ is so large that~$1-\alpha<1/(4k)$, we obtain~$\CE \ll q_1^{-1/4} \Psi(x, y) (H(u)^{-c} + y^{-c})$. If~$y<\e^{\sqrt{\log x}}$, then~$(\log x) =O_\ee(H(u)^{\ee})$ for any~$\ee>0$, so that the required bound
\[ \CE \ll \Psi(x, y)(H(u)^{-c/2}(\log x)^{-A}+y^{-c}) \]
follows immediately from~$q_1\geq 1$. If~$y\geq \e^{\sqrt{\log x}}$, then by Siegel's theorem, we have~$q_1\gg_A (\log Y)^A = (\log x)^{A/2}$ for any~$A>0$ (the constant being ineffective unless~$A<2$). We deduce
\begin{equation}
\CE \ll_A \Psi(x, y)(H(u)^{-c}(\log x)^{-A}+y^{-c}).\label{eq:majo-cE}
\end{equation}

Grouping our bounds~\eqref{eq:majo-cN} and~\eqref{eq:majo-cE}, we have shown
\begin{equation}
E_k(x, y ; a/q) = M_k(x, y ; a/q) + O(\Psi(x, y)(y^{-c}+H(u)^{-c}(\log x)^{-A})),\label{eq:Ek-Mk-aq}
\end{equation}
the implicit constant being effective if~$A=0$.

For general~$\delta$, by integration by parts, we may write
\begin{align*}
E_k(x, y ; \vth) = \e(\delta x^k)E_k(x, y ; a/q) - 2\pi i\delta\int_{x/Y}^x kt^{k-1}\e(\delta t^k)E_k(t, y ; a/q)\dd t + O(\Psi(x/Y, y)).
\end{align*}
The error term here is~$O(\Psi(x, y)/Y^\alpha)$ which is acceptable. Note that for~$t\in[x/Y, x]$, we have~$\log t\asymp \log x$, so that by~\eqref{eq:Ek-Mk-aq}, we have
\[ E_k(t, y ; a/q) = M_k(t, y ; a/q) + O_A(\Psi(t, y)(y^{-c}+H(u)^{-c}(\log x)^{-A})) \qquad (x/Y \leq t\leq x). \]
Note that~$|\delta| \int_{x/Y}^xkt^{k-1}\dd t \leq |\delta|x^k$, so that by~\eqref{eq:arcmaj-1}, we obtain
\begin{align*}
E_k(x, y ; \vth) =\ & \e(\delta x^k)M_k(x, y ; a/q) - 2\pi i\delta\int_{x/Y}^x kt^{k-1}\e(\delta t^k)M_k(t, y ; a/q)\dd t \\
\ &+ O(\Psi(x, y)(1+|\delta|x^k)(y^{-c}+H(u)^{-c}(\log x)^{-A})).
\end{align*}
Integrating by parts, the main terms above are regrouped into
\[ M_k(x, y ; \vth) + O(\Psi(x/Y, y)) \]
which yields our claimed bound.
\end{proof}

The next step is to evaluate the contribution from the principal character~$M_k(x,y; \vth)$. As is classically the case in the study of friable numbers, we shall use two different methods according to the relative sizes of~$x$ and~$y$. 

\subsection{The main term in the neighborhood of~$\vth=0$, for large values of~$y$}\label{sec:main-term-0-big-y}

In this section, we evaluate the contribution of principal characters on the major arc centered at~$0$, when~$y$ is large. The target range for~$(x, y)$ is
\begin{equation}
\label{eq:Hee}\tag{$H_\ee$}
\exp\{(\log\log x)^{5/3+\ee}\} \leq y\leq x.
\end{equation}
Recall that~$\CY_\ee$ is defined in~\eqref{eq:notations}.
\begin{prop}\label{prop:Ek-bigy}
Let~$\ee>0$ be small and fixed. Let~$\delta \in \R$ and write~$\CQ = 1 + |\delta|x^k$. Then whenever~$x$ and~$y$ satisfy~\eqref{eq:Hee}, there holds
\[
E_k(x, y ; \delta) = \Psi(x, y)\left\{ \cPhi(\delta x^k, 1) + O_\ee\Big(\frac{\log(2\CQ)}{\CQ^{1/k}}\cdot \frac{\log(u+1)}{\log y} + \CQ\CY_\ee^{-1}\Big) \right\}.
\]
\end{prop}

\begin{proof}
For~$k=1$, this follows from theorems of La Bret{\`e}che~\cite[Proposition~1]{Breteche-sommes} and La Bret{\`e}che-Granville~\cite[Th\'{e}or\`{e}me~4.2]{BG}. It is based on integration by parts and the theorem of Saias~\cite{Saias}, that
\begin{equation}
\Psi(x, y) = \Lambda(x, y) \big\{1 + O(\CY_\ee^{-1})\big\} \qquad ((x, y) \in \eqref{eq:Hee}).\label{eq:psi-lambda-saias}
\end{equation}
Here De Bruijn's function~$\Lambda(x, y)$ (see~\cite{dB}) is defined by
$$ \Lambda(x, y) := x \int_{-\infty}^\infty \rho(u-v) \dd\Big(\frac{\lfloor y^v \rfloor}{y^v}\Big) \qquad (x\not\in\N) $$
and~$\Lambda(x, y) = \Lambda(x+0, y)$ for~$x\in\N$, where~$\rho$ denotes Dickman's function~\cite[section III.5.3]{ITAN}. This implies in particular the theorem of Hildebrand~\cite{Hildebrand}
\begin{equation}\label{eq:psi-rho}
\Psi(x, y) = x\rho(u) \Big\{1 + O\Big(\frac{\log(u+1)}{\log y}\Big)\Big\} \qquad ((x, y)\in \eqref{eq:Hee}).
\end{equation}

For arbitrary~$k$, the arguments transpose almost identically, so we only sketch the proof. We first use Lemma~\ref{lem:adam-smooth-2} to approximate
$$ E_k(x, y ; \delta) = \ssum{x/\CY_\ee < n\leq x\\P(n)\leq y} \e(n^k\delta) + O(\Psi(x, y)/\CY_\ee^\alpha). $$
The error term here is acceptable. We integrate by parts and use~\eqref{eq:psi-lambda-saias} to obtain
\begin{equation}
\begin{aligned}
\ssum{x/\CY_\ee < n\leq x\\P(n)\leq y} \e(n^k\delta)  = \int_{z=x/\CY_\ee+}^{x+} \e(z^k\delta)\dd(\Lambda(z, y)) + O(\Psi(x, y) \CQ \CY_\ee^{-1}).
\end{aligned}\label{eq:bigy-ipp1}
\end{equation}
For~$z\geq 1$, we let~$F_\delta(z) := \int_0^z\e(\delta t^k)\dd t$ and
\[ \lambda_y(z) := \frac{\Lambda(z, y)}{z} + \frac1{\log y} \int_{-\infty}^{\infty}\rho'\Big(\frac{\log z}{\log y}-v\Big)\dd\Big(\frac{\{y^v\}}{y^v}\Big). \]
Note that~$F_\delta(z) = O(z/(1+z|\delta|^{1/k}))$. Using\cite[p.310, first formula]{BG}, we write
\begin{equation}
\int_{z=x/\CY_\ee+}^{x+} \e(z^k\delta)\dd(\Lambda(z, y)) = \int_{x/\CY_\ee}^x \lambda_y(z)F_\delta'(z)\dd z - \int_{x/\CY_\ee}^x zF_\delta'(z)\dd(\{z\}/z).\label{eq:bigy-lambdalambda}
\end{equation}
By integration by parts, the second integral on the right side in~\eqref{eq:bigy-lambdalambda} is
\[ \big[\{z\}F'_\delta(z)\big]_{z=x/\CY_\ee}^x - \int_{x/\CY_\ee}^x\Big(\frac{F'_\delta(z)}z+F''_\delta(z)\Big)\{z\}\dd z = O\big(\log\CY_\ee + |\delta|x^k\big), \]
and the first integral is
\begin{equation}
\int_{x/\CY_\ee}^x \lambda_y(z)F_\delta'(z)\dd z= \lambda_y(x)F_\delta(x) - \lambda_y(x/\CY_\ee)F_\delta(x/\CY_\ee) - \int_{x/\CY_\ee}^x F_\delta(z)\dd(\lambda_y(z)).\label{eq:bigy-ipp2}
\end{equation}
To evaluate this, we use~\cite[formula (2.3)]{BG} and obtain
\begin{equation}
\begin{aligned}
&\lambda_y(x)F_\delta(x) - \lambda_y(x/\CY_\ee)F_\delta(x/\CY_\ee) = \rho(u)F_\delta(x) + O\Big(\frac{\Psi(x, y)}{\CQ^{1/k}}\frac{\log(u+1)}{\log y} + \Psi(x, y)\CY_\ee^{-\alpha}\Big).
\end{aligned}\label{eq:bigy-estim-mt}
\end{equation}
Next, using~\cite[formula (4.16)]{BG} and integration by parts, we obtain
\begin{equation}
\int_{x/\CY_\ee}^x F_\delta(z)\dd(\lambda_y(z)) = O\Big(\rho(u)\frac{\log(u+1)}{\log y}\int_{x/\CY_\ee}^x\frac{|F_\delta(z)|\dd z}z\Big) +\frac1{\log y}\int_{x/\CY_\ee}^x F_\delta(z)\dd\Big(\frac{\{z/y\}}{z/y}\Big).\label{eq:decomp-terme-discont-Fdelta}
\end{equation}
The integral in the error term is bounded by~$x\log(2\CQ)\CQ^{-1/k}$, and partial summation yields
\begin{align*}
\int_{x/\CY_\ee}^x F_\delta(z)\dd\Big(\frac{\{z/y\}}{z/y}\Big) \ll \min\{y\CY_\ee, x\}\CQ^{-1/k} + \log\CY_{\ee} \ll x\rho(u)\big\{\CQ^{-1/k} + \CY_\ee^{-1}\big\}.
\end{align*}
Inserting into~\eqref{eq:decomp-terme-discont-Fdelta}, we obtain
\begin{equation}
\int_{x/\CY_\ee}^x F_\delta(z)\dd(\lambda_y(z)) \ll x\rho(u) \Big\{ \frac{\log(2\CQ)}{\CQ^{1/k}}\frac{\log(u+1)}{\log y} + \CY_\ee^{-1} \Big\}.\label{eq:bigy-estim-resid}
\end{equation}
Combining the estimates~\eqref{eq:bigy-estim-resid},~\eqref{eq:bigy-estim-mt},~\eqref{eq:bigy-ipp2} and~\eqref{eq:bigy-ipp1}, we obtain
\begin{equation}
E_k(x, y ; \vth) = x\rho(u)\Big\{\frac{F_\delta(x)}x + O\Big(\frac{\log(2\CQ)}{\CQ^{1/k}}\frac{\log(u+1)}{\log y} + \CQ\CY_\ee^{-\alpha}\Big)\Big\}.\label{eq:bigy-Ek-1}
\end{equation}
Using~\eqref{eq:psi-rho} and rescaling~$\ee$ completes the argument.
\end{proof}

\subsection{The main term in the neighborhood of~$\vth=0$, for small values of~$y$}\label{sec:main-term-0-small-y}

For smaller values of~$y$, we employ the saddle-point method~\cite{HT} based on exploiting the nice analytic behaviour of the Mellin transform
\[ \zeta(s, y) := \prod_{p\leq y}(1-p^{-s})^{-1} \]
associated with the set of~$y$-friable integers. By Perron's formula,
\[ E_k(x, y ; \delta) = \frac1{2\pi i}\int_{\kappa-i\infty}^{\kappa+i\infty}\zeta(s, y)\cPhi(\delta x^k, s)x^s\frac{\dd s}s \qquad (x\not\in\N), \]
where~$\kappa>0$ is arbitrary. The saddle-point~$\alpha = \alpha(x, y)$, defined in terms of~$x$ and~$y$ by means of the implicit equation~\eqref{eq:def-alpha}, is the unique positive real number~$\sigma$ achieving the infimum~$\inf_{\sigma>0}x^\sigma\zeta(\sigma, y)$. Recall the definition of~$\CT_\ee$ from~\eqref{eq:notations}.

\begin{prop}\label{prop:Ek-smally}
Let~$\ee > 0$ be small and fixed. Let~$\delta \in \R$ and write~$\CQ = 1 + |\delta|x^k$. Then whenever~$x$ and~$y$ satisfy~$(\log x)^{1+\ee}\leq y\leq x$, there holds
\[ E_k(x, y ; \delta) = \Psi(x, y)\Big\{\cPhi(\delta x^k, \alpha) + O\Big(\frac1{\CQ^{\alpha/k-\ee}} \cdot \frac1u + \CQ \CT_\ee^{-c}\Big)\Big\}, \]
for some constant~$c > 0$.
\end{prop}

\begin{proof}
One option is to transpose the arguments of~\cite[Proposition~2.11]{D2014}. Instead we take a simpler route, inspired from a remark of D. Koukoulopoulos. When~$y>x^{1/(\log \log x)^2}$, we have~$1 - \alpha \ll 1/u$ by~\eqref{eq:alpha-1}, and thus the estimate is a consequence of Proposition~\ref{prop:Ek-bigy} since
\begin{equation}\label{eq:cPhi-1} 
\cPhi(\delta x^k,\alpha) - \cPhi(\delta x^k, 1) \ll (1-\alpha) \frac{\log2 \CQ}{\CQ^{\alpha/k}} 
\end{equation}
by Lemma~\ref{lem:phi-check}.

We assume henceforth that~$y\leq x^{1/(\log\log x)^2}$, with the consequence that~$\log x \ll_\ee H(u)^\ee$. Using Lemma~\ref{lem:adam-smooth-2}, we write
\begin{equation}
\begin{aligned}
  E_k(x, y ; \vth) = \int_{x/\CT_\ee}^x\e(\delta t^k)\dd(\Psi(t, y)) + O(\Psi(x, y)\CT_\ee^{-\alpha}).
\label{eq:Ek-IPP-gdy}
\end{aligned}
\end{equation}
Let~$\alpha_t := \alpha(t, y)$ and~$u_t := (\log t)/\log y$. Then for~$t\in[x/\CT_\ee, x]$, by~\cite[Lemma 10]{HT} we have
\[ \Psi(t, y) = \frac1{2\pi i}\int_{\alpha_t-i/\log y}^{\alpha_t+i/\log y}\zeta(s, y)\frac{t^s\dd s}{s} + O\Big(t^{\alpha_t}\zeta(\alpha_t, y)\big\{\e^{-(\log y)^{3/2-\ee}}+H(u_t)^{-c}\big\}\Big). \]
Note that~$\log \CT_\ee \ll u/(\log u)^2$, so that certainly~$u_t = u + O(u/(\log y)) \asymp u$, and thus~$H(u_t)^{-c} \ll H(u)^{-c'}$.  On the other hand, from~\eqref{eq:HT-psixy}, \eqref{eq:sigma2}, and Lemma~\ref{lem:adam-smooth-2} we have
\[ t^{\alpha_t}\zeta(t, y) = O(\Psi(t, y)\log x) = O\Big( \Big(\frac{t}{x}\Big)^{\alpha} \Psi(x,y) \log x \Big) .\]
By our assumption that~$(\log x)^{1+\ee} \leq y\leq x^{1/(\log\log x)^2}$, we can absorb the~$\log x$ factor into the error terms and obtain
\begin{equation}\label{eq:psity} 
\Psi(t, y) = \frac1{2\pi i}\int_{\alpha_t-i/\log y}^{\alpha_t+i/\log y}\zeta(s, y)\frac{t^s\dd s}{s} + O\Big( \Big(\frac{t}{x}\Big)^{\alpha} \Psi(x, y)\CT_\ee^{-c}\Big). 
\end{equation}
We now shift the contour of integration to the line between~$\alpha\pm i/\log y$. For~$t\in[x/\CT_\ee, x]$, by \eqref{eq:sigma2} we have
\[ \sigma_2(\alpha_t, y) \asymp (\log x)\log y \asymp \sigma_2(\alpha, y). \]
By~\cite[Lemma 8.(i)]{HT}, we therefore have
\[ \Big|\frac{\zeta(\alpha + i/\log y, y)}{\zeta(\alpha, y)}\Big| \leq \e^{-c u}. \]
This implies
\begin{equation}\label{eq:psity-main} 
\frac1{2\pi i}\int_{\alpha_t-i/\log y}^{\alpha_t+i/\log y}\zeta(s, y)\frac{t^s\dd s}{s} = \frac1{2\pi i}\int_{\alpha-i/\log y}^{\alpha+i/\log y}\zeta(s, y)\frac{t^s\dd s}{s} + O\Big((\alpha_t-\alpha)\e^{-cu}\frac{t^\alpha\zeta(\alpha, y)}{\alpha}\Big). 
\end{equation}
Here, we have used the bound~$\sup_{\beta\in[\alpha, \alpha_t]}t^{\beta}\zeta(\beta, y)\leq t^{\alpha}\zeta(\alpha, y)$ which follows by unimodality and the definition of the saddle-point. If we view~$\alpha_t$ as a function of~$u_t$, then
\[ \frac{\dd\alpha_t}{\dd u_t} = - \frac{\log y}{\sigma_2(\alpha_t,y)} \]
by the definition of~$\sigma_2$ and the saddle point~$\alpha_t$. It thus follows from \eqref{eq:sigma2} that
\[ \alpha_t - \alpha \leq (u_t - u) \sup_t \frac{\log y}{|\sigma_2(\alpha_t, y)|} \ll \frac{\log\CT_\ee}{\log y} \cdot \frac{1}{\log x}.  \]
Using \eqref{eq:HT-psixy} and \eqref{eq:sigma2} to bound~$\zeta(\alpha,y)$, we deduce
\[ (\alpha_t-\alpha)\e^{-cu}\frac{t^\alpha\zeta(\alpha, y)}{\alpha} \ll \frac{\log\CT_\ee}{\log y} \cdot \frac{1}{\log x}  \cdot \e^{-cu}\Big(\frac{t}{x}\Big)^{\alpha}\Psi(x, y) \log x \ll \Big(\frac tx\Big)^\alpha\Psi(x, y)\CT_\ee^{-c'}.  \]
Inserting this into~\eqref{eq:psity} and~\eqref{eq:psity-main}, we obtain
\[ \Psi(t, y) = \frac1{2\pi i}\int_{\alpha-i/\log y}^{\alpha+i/\log y}\zeta(s, y)\frac{t^s\dd s}{s} + O\Big(\Big(\frac tx\Big)^\alpha\Psi(x, y)\CT_\ee^{-c}\Big). \]
We insert this estimate into~\eqref{eq:Ek-IPP-gdy} and integrate by parts to obtain
\begin{equation}
\label{eq:Ek-petitsy-ap-ipp}
E_k(x, y ; \vth) = \frac1{2\pi i}\int_{\alpha-i/\log y}^{\alpha+i/\log y}\zeta(s, y)\int_{x/\CT_\ee}^{x}\e(\delta t^k)t^{s-1}\dd t\dd s  + O\big(\Psi(x, y)\CQ\CT_\ee^{-c}\big).
\end{equation}
Note that
\[ \int_{x/\CT_\ee}^{x}\e(\delta t^k)t^{s-1}\dd t = \int_0^x\e(\delta t^k)t^{s-1}\dd t + O\big((x/\CT_\ee)^\alpha\big) = \frac{x^s}s\cPhi(\delta x^k,s) + O\big((x/\CT_\ee)^\alpha\big). \]
The contribution to~$E_k(x,y;\vth)$ from the error term~$O\big((x/\CT_\ee)^\alpha\big)$ above is bounded by
$$ \frac{\zeta(\alpha, y)x^\alpha}{(\log y)\CT_\ee^\alpha} \ll \Psi(x, y) \CT_\ee^{-c}. $$
 Therefore,
\[ E_k(x, y ; \delta) = \frac1{2\pi i}\int_{\alpha-i/\log y}^{\alpha+i/\log y}\zeta(s, y)\cPhi(\delta x^k, s)x^s\frac{\dd s}s
+ O(\Psi(x, y) \CQ \CT_\ee^{-c}). \]
The evaluation of the remaining integral can now be done as in~\cite[Proposition~2.11]{D2014} (in particular the treatment of segment~$\CC_4$ on p.623), by splitting the integral depending on the size of the imaginary part of~$s$ relative to~$T_0 := (u^{1/3}\log y)^{-1}$. Large values of~$|\tau|$ are handled using~\cite[Lemma~8.(i)]{HT}, while the contribution of small values of~$|t|$ is estimated by a Taylor formula at order~$4$. After some routine calculations, we find
\[ \frac1{2\pi i}\int_{\alpha-i/\log y}^{\alpha+i/\log y}\zeta(s, y)\cPhi(\delta x^k, s)x^s\frac{\dd s}s = \Psi(x, y)\cPhi(\delta x^k, \alpha) + O\Big(\frac{\Psi(x, y)}{\CQ^{\alpha/k-\ee}} \cdot \frac1{u}\Big). \]
This concludes the proof.
\end{proof}

\subsection{The main term for general major arcs}\label{sec:main-term-general}

In this section we estimate the main term~$M_k(x, y ; \vth)$ (defined in \eqref{Mk}) in all of the major arcs, using the estimates proved in the previous two sections. This mirrors analogous calculations in~\cite[Section A.2]{Harper}. We recall the notations in \eqref{eq:notations}.

\begin{prop}\label{prop:estim-Mk}
Let~$\ee>0$ be small and fixed. Let~$2\leq y\leq x$ be large, and let~$\vth=a/q + \delta$ with~$0 \leq a < q\leq Y^\eta$ for some sufficiently small~$\eta > 0$ and~$(a,q) = 1$. Write~$\CQ = q(1 + |\delta|x^k)$.
\begin{enumerate}
\item Whenever~$x$ and~$y$ satisfy~\eqref{eq:Hee}, we have
\[
\frac{M_k(x, y ; \vth)}{\Psi(x,y)} = \cPhi(\delta x^k, 1)H_{a/q}(1) + O\Big(\frac{q^{1-\alpha}}{\CQ^{1/k-\ee}} \cdot \frac{\log(u+1)}{\log y} + \CQ \CY_\ee^{-1}\Big). \]
\item Whenever~$x$ and~$y$ satisfy~$(\log x)^{1+\ee}\leq y\leq x$, we have
\[
\frac{M_k(x, y ; \vth)}{\Psi(x,y)} = \cPhi(\delta x^k, \alpha)H_{a/q}(\alpha) + O\Big(\frac{q^{1-\alpha}}{\CQ^{\alpha/k-\ee}} \cdot \frac{1}{u} + \CQ \CT_\ee^{-c}\Big),
\]
for some constant~$c > 0$.
\end{enumerate}
\end{prop}
\begin{proof}
We only give the details of deducing the first part of the statement from~\eqref{Mk} and Proposition~\ref{prop:Ek-bigy}; the proof of the second part is similar, using Proposition~\ref{prop:Ek-smally} instead. Write~$\CQ' = 1 + |\delta|x^k$ so that~$\CQ = q \CQ'$.
Since~$q\leq Y^\eta$, we have~$\log(x/q)\asymp \log x$, so that for each~$d_1,d_2$ with~$d_1d_2 \mid q$ and~$P(d_1d_2) \leq y$, we can apply Proposition~\ref{prop:Ek-bigy} and obtain
\[
E_k(x/(d_1d_2), y ; (d_1d_2)^k\delta) = \Psi\Big(\frac{x}{d_1d_2}, y\Big)\Big\{ \cPhi(\delta x^k, 1)
+ O\Big(\frac{1}{\CQ'^{1/k-\ee}} \cdot \frac{\log(u+1)}{\log y} + \CQ' \CY_\ee^{-1} \Big) \Big\}.
\]
By~\cite[Th{\'e}or{\`e}me~2.4]{BT2005} we have, uniformly for~$d_1d_2\leq q\leq y^\eta$,
\[ \Psi\Big(\frac{x}{d_1d_2}, y\Big) = \frac{\Psi(x, y)}{(d_1d_2)^\alpha}\Big(1 + O\Big((\log q)\frac{\log(u+1)}{\log y}\Big)\Big). \]
Combining this with the bounds~$\Psi(x/(d_1d_2), y) \ll (d_1d_2)^{-\alpha} \Psi(x, y)$ from Lemma~\ref{lem:adam-smooth-2} and~$\cPhi(\delta x^k, 1) \ll (1+|\delta|x^k)^{-1/k}$ from Lemma~\ref{lem:phi-check}, we deduce
\[
E_k(x/(d_1d_2), y ; (d_1d_2)^k\delta) = \frac{\Psi(x, y)}{(d_1d_2)^\alpha}\Big\{\cPhi(\delta x^k, 1)
+ O\Big(\frac{\log q}{\CQ'^{1/k-\ee}} \cdot \frac{\log(u+1)}{\log y} + \CQ' \CY_\ee^{-1}\Big)\Big\}.
\]
Inserting this estimate into~\eqref{Mk} and recalling the definition of~$H_{a/q}(\alpha)$ in~\eqref{eq:def-Hq}, we obtain
\begin{equation}\label{eq:Mk-final}
\frac{M_k(x, y ;\vth)}{\Psi(x,y)} = \cPhi(\delta x^k, 1)H_{a/q}(\alpha) + O\Big(\Big(\frac{1}{\CQ'^{1/k-\ee}} \cdot \frac{\log(u+1)}{\log y} + \CQ' \CY_\ee^{-1}\Big) \CR \Big),
\end{equation}
where
\[ \CR := \sum_{\substack{d_1d_2 \mid q \\ P(d_1d_2) \leq y}}\frac{(\log q)}{(d_1d_2)^\alpha\vphi(q/d_1)}|G_k(q/d_1, ad_1^{k-1}, \chi_0)|. \]
Using the bound~\eqref{eq:borne-Gk-nontv} with~$q_1 = 1$ (a consequence of Lemma~\ref{lemme:gauss}) to bound the Gauss sum~$G_k(q/d_1, ad_1^{k-1}, \chi_0)$ above by~$q^{1-1/k+\ee}$, we obtain
\begin{equation}
\CR \ll q^{1-1/k+\ee}\sum_{d_1d_2|q}\frac{(d_1d_2)^{-\alpha}}{\vphi(q/d_1)} \ll q^{1-\alpha-1/k+\ee}.\label{eq:cR}
\end{equation}
Finally, to see that~$H_{a/q}(\alpha)$ is close to~$H_{a/q}(1)$, note that the derivative~$H_{a/q}'$ satisfies the bound~$H_{a/q}'(\sigma)=O(q^\ee\CR)$ for all~$\sigma\in[\alpha, 1]$. Thus from~\eqref{eq:alpha-1}  we obtain
\begin{equation}
H_{a/q}(\alpha) = H_{a/q}(1) + O(q^\ee\CR\log(u+1)/\log y).\label{eq:Haq-1}
\end{equation}
In view of \eqref{eq:cR} and Lemma~\ref{lem:phi-check}, we may replace~$H_{a/q}(\alpha)$ in \eqref{eq:Mk-final} by~$H_{a/q}(1)$ at the cost of an acceptable error. This completes the proof.
\end{proof}

\subsection{Deduction of Theorem~\ref{thm:estim-majo}}

Let the situation be as in the statement of Theorem~\ref{thm:estim-majo}. By choosing~$C$ large enough, we may assume that the hypotheses of Propositions~\ref{prop:Ek-Mk} and~\ref{prop:estim-Mk} are satisfied, and moreover that the error term in~\eqref{eq:Ek-Mk} is acceptable. We divide into two cases depending on whether to apply the first or the second part of Proposition~\ref{prop:estim-Mk}.

Assume first that~$\e^{\sqrt{\log x \log\log x}} \leq y$. Then~$1/u \gg \log(u+1)/\log y$ and~$\log x \ll \CY_\ee^{o(1)}$, so that the error term in the first part of Proposition~\ref{prop:estim-Mk} is acceptably small. To see that we may replace~$\cPhi(\delta x^k,1) H_{a/q}(1)$ by~$\cPhi(\delta x^k,\alpha) H_{a/q}(\alpha)$, note that by~\eqref{eq:cPhi-1} and~\eqref{eq:Haq-1} again, we have
$$ \cPhi(\delta x^k, 1)H_{a/q}(1) = \cPhi(\delta x^k, \alpha)H_{a/q}(\alpha) + O\Big(\frac{q^{1-\alpha}}{\CQ^{\alpha/k-\ee}} \cdot \frac{\log(u+1)}{\log y}\Big). $$
This error term is again acceptable.

Assume next that~$(\log x)^{CA} \leq y \leq \e^{\sqrt{\log x\log \log x}}$. Then~$1/u \ll \log(u+1)/\log y$ and~$\log x \ll \CT_\ee^{o(1)}$, so that the error term in the second part of Proposition~\ref{prop:estim-Mk} is acceptably small, and the conclusion follows.

Finally, the upper bound~\eqref{eq:arcmaj-majoration} follows from Lemmas~\ref{lem:phi-check} and~\ref{lem:Haq-bound}.

\section{Minor arc estimates}\label{sec:minor}

The goal of this section is to prove Theorem~\ref{thm:estim-minor}. It is convenient to prove the following equivalent form. For parameters~$Q, X\geq 1$ and~$0\leq a \leq q\leq Q$ with~$(a,q)=1$, define
\[ \major(q,a;Q,X) = \{ \vth \in [0,1]: |q\vth -a | \leq QX^{-k} \}, \]
and
\begin{equation}
\major(Q,X) := \bigcup_{\substack{0 \leq a < q \leq Q \\ (a, q)=1}} \major(q, a ; Q, X).\label{eq:def-majorarcs}
\end{equation}
In particular, for any~$\vth = a/q+\delta$ with~$0 \leq a \leq q$ and~$(a,q) = 1$, we must have~$q(1+|\delta| X^k) \geq Q$ whenever~$\vth \notin \major(Q,X)$. Note also that we have the obvious inclusion~$\major(Q_1, X) \subset \major(Q_2, X)$ whenever~$Q_1 \leq Q_2$. 

\begin{prop}\label{prop:estim-minor}
Fix a positive integer~$k$.  There exists~$K = K(k) > 0$ and~$c = c(k) > 0$ such that the following statement holds. Let~$2 \leq y \leq x$ be large with~$y \geq (\log x)^K$. If~$\vth\in[0, 1]\setminus\major(Q, x)$ for some~$Q \geq 1$, then
$$ E_k(x, y ; \vth) \ll \Psi(x, y) Q^{-c}. $$
\end{prop}

\begin{proof}[Proof that Proposition~\ref{prop:estim-minor} implies Theorem~\ref{thm:estim-minor}]
We may assume that~$10 \leq q \leq 0.1 x^k$, since otherwise the claim is trivial. Let~$Q = (1/3)\min(q, \sqrt{x^k/q})$. In view of Theorem~\ref{thm:estim-minor}, it suffices to show that~$\vth \notin \major(Q, x)$. Suppose, on the contrary, that~$\vth = a'/q' + \delta$ for some~$0 \leq a' \leq q' \leq Q$ with~$(a',q') = 1$, and~$|\delta| \leq Qx^{-k}$. Then by our choice of~$Q$ we have
\[ q \geq 3Q \geq 3q', \ \ \frac{Q}{x^k} \leq \frac{1}{9qQ} \leq \frac{1}{9qq'}. \]
Hence
\[ \left| \frac{a}{q} - \frac{a'}{q'} \right| \leq \frac{1}{q^2} + |\delta| \leq \frac{1}{3qq'} + \frac{1}{9qq'} < \frac{1}{qq'}.   \]
It follows that~$a=a'$ and~$q=q'$, but this is impossible since~$q \geq 3Q$ and~$q' \leq Q$.
\end{proof}

The bulk of the proof of Proposition~\ref{prop:estim-minor} lies in Section~\ref{sec:minor-small-y}, which applies when~$y = (\log x)^K$ for some constant~$K$. In Sections~\ref{sec:minor-complete} and~\ref{sec:minor-large-y}, we quote and prove some complimentary results valid for larger~$y$.

\subsection{Estimates for complete Weyl sums}\label{sec:minor-complete}

We start with the following estimate for complete Weyl sums.

\begin{lemma}\label{lemme:Weyl-qpetit}
Fix a positive integer~$k$. Let~$x$ be large, and let~$\vth = a/q + \delta$ for some~$0 \leq a \leq q$ and~$(a,q) = 1$. Assume that~$|\delta|\leq 1/(qx)$, and write~$\CQ = q(1 + |\delta|x^k)$. Then
\[
\Big|\sum_{n\leq x} \e(\vth n^k)\Big| \ll x\Big(\frac1x + \frac q{x^k}+ \frac{1}{\CQ} \Big)^{\sigma(k)}
\]
for some~$\sigma(k)>0$.
\end{lemma}

Compared with classical estimates, the bound here decays not only with~$q$ but also with~$\delta$. This will be necessary in the proof of Proposition~\ref{estim-verysmooth-minor} below. The extra dependence on~$\delta$ can be easily obtained by following the standard Weyl differencing argument, which was done in \cite[Lemma 4.4]{GT-manifold}. In fact, Lemma~\ref{lemme:Weyl-qpetit} is nothing but a reformulation of \cite[Lemma 4.4]{GT-manifold}.

\begin{proof}
Let~$D = 0.1 \min(x, x^k/q, \CQ)$. If the desired exponential sum estimate fails, then by \cite[Lemma 4.4]{GT-manifold}, there is a positive integer~$d \leq D$ such that~$\|d \vth\| \leq D/x^k$.  By the choice of~$D$ and the assumption on~$\delta$, we have 
\[ |d \delta| \leq D|\delta| \leq 0.1 x |\delta| < 1/(2q). \]
In the case when~$q \nmid d$, we have
\[ \|d \vth\| \geq 1/q - |d \delta| > 1/(2q) > D/x^k, \]
where the last inequality follows again from the choice of~$D$. This is a contradiction.
In the case when~$q \mid d$, we have~$\|d \vth\| = |d \delta|$ and~$d \geq q$. This is again a contradiction since~$|d \delta| \geq |q \delta| > D/x^k$ by the choice of~$D$ and the definition of~$\CQ$.
\end{proof}

\begin{remark}
To get a better exponent~$\sigma(k)$ in the statement above, one should follow Vaughan's treatment \cite[Section 5]{Vaughan97} using works on the Vinogradov main conjecture, which has recently been proved (trivial for~$k = 1, 2$, in the case~$k=3$ by Wooley~\cite{Wooley-cubic}, and for all~$k>3$ by Bourgain--Demeter--Guth~\cite{BDG-Decoupling}). We will not pursue this further.
\end{remark}

\subsection{Friable Weyl sums, for large values of~$y$}\label{sec:minor-large-y}

The following minor arcs estimate due to Wooley~\cite[Theorem~4.2]{Woo95} is useful for mildly friable numbers.

\begin{prop}
\label{prop:weyl-smooth-wooley}
Fix a positive integer~$k$ and some~$\lambda\in(0, 1]$. There exist~$\eta, \sigma>0$, depending on~$k$ and~$\lambda$, such that the following holds. Let~$2 \leq y \leq x$ be large with~$y \leq x^{\eta}$, and let~$\vth \in [0,1] \setminus \major(x^{\lambda}, x)$. Then~$E_k(x,y; \vth) \ll x^{1-\sigma}$.
\end{prop}

\begin{proof}
This follows from~\cite[Theorem~1.1]{Woo95} when~$\lambda = 1$. In the general case, this follows from \cite[Theorem 4]{Wooley-exp-sum-smooth-general}.
\end{proof}

The following proposition covers the range~$x^{\eta} \leq y \leq x$. In its proof we adopt the natural strategy of factoring out largest prime factors of non-$y$-friable numbers.

\begin{prop}\label{prop:weyl-large-y}
Fix a positive integer~$k$ and some~$\eta\in(0, 1]$. Let~$2 \leq y \leq x$ be large with~$y \geq x^{\eta}$, and let~$\vth \in [0,1] \setminus \major(Q, x)$ for some~$Q \geq 1$. Then~$E_k(x, y; \vth) \ll x Q^{-c}$ for some~$c=c(k, \eta)>0$.
\end{prop}

\begin{proof}
When~$\eta = 1$ the conclusion follows from Lemma~\ref{lemme:Weyl-qpetit}. Now assume that the conclusion holds when~$\eta \geq 1/s$ for some positive integer~$s$, and let~$\eta \in [1/(s+1), 1/s)$.  We may write
\[ E_k(x, y ; \vth) = E_k(x, x^{1/s} ; \vth) - \sum_{y < p \leq x^{1/s}} \sum_{n \in S(x/p, p)} e((pn)^k \vth). \]
The bound~$|E_k(x, x^{1/s}; \vth)| \ll x Q^{-c}$ follows from the induction hypothesis. To treat the double sum, split it into dyadic intervals so that we need to prove
\begin{equation}\label{eq:S(P)} 
S(P) = \sum_{P<p\leq 2P} \sum_{n \in S(x/p, p)} e( (pn)^k \vth) \ll xQ^{-c}, 
\end{equation}
for~$y \leq P \leq x^{1/s}$. We divide into two cases depending on whether~$P\leq x/Q^c$ or not (in fact,~$P \leq x/Q^c$ is the only case unless~$s = 1$).

Assume first~$P \leq x/Q^c$ so that~$x/P \geq Q^c$. We bound~$S(P)$ by
\[ S(P) \leq \sum_{P<m\leq 2P} \Big| \sum_{n \in S(x/m, m)} e( (mn)^k \vth) \Big|. \]
Here we have dropped the primality condition on~$m$.  Let~$R \geq 1$ be a parameter that will be chosen to be a small power of~$Q$, and let~$\CM$ be the set of~$m\in(P, 2P]$ with~$m^k \vth \in \major(R, x/m)$. Since~$m \geq (x/m)^{1/s}$, we may apply the induction hypothesis to the inner sum when~$m \notin \CM$ to obtain
\[ S(P) \ll \frac xP |\CM| + xR^{-c} . \]
To complete the proof of \eqref{eq:S(P)} in this case, it suffices to show that~$|\CM| \ll PR^{-1}$. Suppose, for the sake of contradiction, that~$|\CM| \geq P R^{-c}$. For each~$m \in \CM$, we may find~$q_m \leq R$ such that~$\| m^k (q_m\vth) \| \leq R (x/m)^{-k}$. By the pigeonhole principle, there exists~$q_0 \leq R$ such that~$\| m^k (q_0 \vth) \| \ll RP^k/x^k$ for at least~$|\CM|/R$ values of~$m \in \CM$.

Now apply Lemma~\ref{lem:vinogradov-polynomial} to the angle~$q_0 \vth$ with~$\varepsilon = R(x/P)^{-k} \leq RQ^{-c}$ and~$\delta = |\CM|/(R P) \geq R^{-2}$. Since~$\varepsilon < \delta/5$ if~$R$ is a sufficiently small power of~$Q$, we conclude that  there is a positive integer~$q \ll \delta^{-O(1)} \ll R^{O(1)}$ such that 
\[ \| q q_0 \vth \| \ll \frac{\delta^{-O(1)}\varepsilon}{P^k} \ll \frac{R^{O(1)}}{x^k}. \]
This contradicts the assumption that~$\vth \notin \major(Q, x)$ if~$R$ is a sufficiently small power of~$Q$.

It remains to deal with the case when~$P \geq x/Q^c$ (which only happens when~$s = 1$). From the assumption~$\vth \notin \major(Q, x)$ we may deduce that for all~$n\leq Q^c$, we have~$n^k\vth \notin \major(Q^{1/2}, 2P)$. Bounding~$S(P)$ in \eqref{eq:S(P)} by
\[ S(P) \leq \frac xP \sup_{n\leq x/P} \Big| \sum_{P < p \leq \min\{2P, x/n\}} e( (pn)^k \vth ) \Big|,  \]
the conclusion follows from estimates for Weyl sums over primes stated below.
\end{proof}

\begin{lemma}
Fix a positive integer~$k$. Let~$x$ be large, and let~$\vth \in [0,1] \setminus \major(Q, x)$ for some~$Q \geq 1$. Then
\[ \left| \sum_{p \leq x} \e(p^k \vth) \right| \ll x Q^{-c} \] 
for some~$c=c(k)>0$.
\end{lemma}

\begin{proof}
We may assume that~$Q \geq (\log x)^A$ for some large constant~$A$, since otherwise the statement is trivial. If~$\vth \in \major(x^{0.1}, x)$, then the conclusion follows from~\cite[Theorem~2]{Kum}. Now assume that~$\vth \notin \major(x^{0.1}, x)$. By Diophantine approximation, we may find~$0 \leq a \leq q \leq x^{k-0.1}$ with~$(a,q) = 1$ such that~$|q\vth - a|\leq x^{-k+0.1}$. Since~$\vth \notin \major(x^{0.1},x)$, we have~$q \geq x^{0.1}$. The  conclusion then follows from a standard minor arc bound such as
\[ \left| \sum_{p \leq x} \e(p^k \vth) \right| \ll x^{1+\ee}(q^{-1} + x^{-1/2} + qx^{-k})^{4^{1-k}} \]
from \cite{Harman}.
\end{proof}

\subsection{Friable Weyl sums, for small values of~$y$}\label{sec:minor-small-y}

Note that Proposition~\ref{prop:weyl-smooth-wooley} does not apply to ~$\vth$ in minor arcs when~$q$ and~$|\delta|x^k$ grow slower than any positive power of~$x$. In this section, we take care of this situation  by a variant of Vinogradov's method, roughly following the argument of Harper~\cite{Harper}.

\begin{prop}\label{estim-verysmooth-minor}
Fix a positive integer~$k \geq 2$. Let~$2 \leq y \leq x$ be large and let~$\alpha = \alpha(x,y)$. Let~$\vth = a/q+\delta$ for some~$0 \leq a \leq q$ and~$(a, q)=1$. Write~$\CQ = q(1 + |\delta|x^k)$, and assume that~$4y^2 \CQ^3 \leq x$. Then for some~$\sigma=\sigma(k)>0$, we have
\[
E_k(x,y; \vth) \ll \Psi(x, y) \CQ^{-\sigma+2(1-\alpha)} (\log x)^5.
\]
\end{prop}

\begin{proof}
We may assume that~$y \geq (\log x)^6$, since otherwise the claim is trivial by taking~$\sigma < 1/6$. Extracting the gcd~$d=(n, q^\infty)$, we may write
\[ E_k(x,y; \vth) = \sum_{\substack{d \mid q^{\infty} \\ P^+(d) \leq y}} \sum_{\substack{n \leq x/d \\ P^+(n) \leq y \\ (n,q)=1}} \e( (nd)^k \vth). \]
The contribution from those terms with~$d \geq \CQ$ is bounded by
\[ \ssum{d \mid q^\infty \\ d \geq \CQ} \Psi(x/d, y) \ll \Psi(x,y) \ssum{d \mid q^\infty \\ d \geq \CQ} d^{-\alpha} \ll_{\varepsilon} \CQ^{-\alpha + \varepsilon} \Psi(x,y),  \]
where the first inequality follows from Lemma~\ref{lem:adam-smooth-2} and the second inequality follows by Rankin's trick. Hence
\[ E_k(x, y ; \vth) =\ \ssum{d|q^\infty\\d\leq \CQ\\P^+(d)\leq y}\ssum{n\leq x/d\\P^+(n)\leq y\\(n, q)=1} \e((nd)^k\vth) + O\Big(\frac{\Psi(x, y)}{\CQ^{\alpha/2}}\Big). \]
We may also discard the terms with~$n \leq x/\CQ$ from the above, since their contribution is bounded by
\[ \ssum{d \mid q^\infty \\ d \leq \CQ} \Psi(x/\CQ, y) \ll \CQ^{-\alpha} \Psi(x,y) \ssum{d \mid q^\infty \\ d \leq \CQ}  \ll_{\varepsilon} \CQ^{-\alpha + \varepsilon} \Psi(x,y), \]
where, again, the first inequality follows from Lemma~\ref{lem:adam-smooth-2} and the second from Rankin's trick. It follows that
\[ E_k(x, y ; \vth) =\ \ssum{d|q^\infty\\d\leq \CQ\\P^+(d)\leq y}\ssum{x/\CQ<n\leq x/d\\P^+(n)\leq y\\(n, q)=1} \e((nd)^k\vth) + O\Big(\frac{\Psi(x, y)}{\CQ^{\alpha/2}}\Big). \]
Let~$L = 4y \CQ$ be a parameter. For the inner sum over~$n$, factoring out a divisor~$m$ of size about~$L$ by taking the product of the smallest prime factors of~$n$, we may write
\[
E_k(x, y ; \vth) 
=\ \ssum{d|q^\infty\\d\leq \CQ\\P^+(d)\leq y}\summ_{\substack{L < m \leq P^+(m)L\\x/(m\CQ)<n\leq x/(md)\\P^+(m)\leq P^-(n)\\P^+(n)\leq y \\ (mn,q) = 1}}\e((mnd)^k\vth) + O\Big(\frac{\Psi(x, y)}{\CQ^{\alpha/2}}\Big),
\]
which is allowed by our hypothesis~$yL \leq x/\CQ$. For~$M\in[L, yL]$, define
\[ \CE(M) := \ssum{d|q^\infty\\d\leq \CQ\\P^+(d)\leq y}\summ_{\substack{M < m \leq\min\{2M, P^+(m)L\}\\x/(m\CQ)<n\leq x/(md)\\P^+(m)\leq P^-(n)\\P^+(n)\leq y \\ (mn,q) = 1}}\e((mnd)^k\vth) .\]
Now mover the sum over~$n$ inside, and bound this inner sum by its absolute value. It is also convenient to remove the dependence on~$m$ in the condition~$x/(m\CQ) < n \leq x/(md)$, which can be achieved by a standard Fourier analytic argument. We obtain
\[ \CE(M) \ll (\log x)\sup_{\beta\in[0, 1)} \ssum{d|q^\infty\\d\leq \CQ}\ssum{M<m\leq 2M\\P^+(m)\leq y} \Big|\ssum{x/(2M\CQ)<n\leq x/(Md)\\P^+(n)\leq y\\P^-(n)\geq P^+(m) \\ (n,q) = 1} \e((mnd)^k\vth + \beta n)\Big|. \]
By the Cauchy--Schwarz inequality and factoring out the largest prime factor~$p = P^+(m)$ of~$m$, we deduce that for some~$\beta\in[0, 1)$,
\begin{equation}
\CE(M) \ll_{\ee} (\log x) \CQ^\ee M^{1/2}\CS_1(M)^{1/2}\label{eq:etape-CS}
\end{equation}
for any~$\ee > 0$, where
\[ \CS_1(M) := \ssum{d|q^\infty\\d\leq \CQ} \sum_{p\leq y} \ssum{M/p<m\leq 2M/p}\Bigg|\ssum{x/(2M\CQ)< n \leq x/(Md)\\P^-(n)\geq p, \ P^+(n)\leq y\\ (n, q)=1}\e((pmnd)^k\vth + \beta n)\Bigg|^2. \]
After expanding the squares and switching the order of summation, we obtain
\[ \CS_1(M) \ll \ssum{d|q^\infty\\d\leq \CQ} \sum_{p\leq y}\ssum{x/(2M\CQ)< n_1 \leq n_2 \leq x/(Md)\\(n_i, q)=1, \ P^+(n_i)\leq y} \Big| \ssum{M/p < m \leq 2M/p} \e((pmd)^k\vth(n_1^k-n_2^k))\Big|. \]
By the hypotheses and the choice of~$L$, it is straightforward to verify that
\[ \left| (pd)^k \delta (n_1^k - n_2^k) \right|  \leq \frac{p}{2qM} \]
for~$1 \leq n_1, n_2 \leq x/(Md)$. Thus we may apply Lemma~\ref{lemme:Weyl-qpetit} and obtain
\[ \ssum{M/p < m \leq 2M/p} \e((mdp)^k\vth(n_1^k-n_2^k)) \ll \frac{M}{p} \cdot \frac{(q, (pd)^k(n_2^k-n_1^k))^\sigma}{(q(1+|\delta| (Md)^k(n_2^k-n_1^k)))^\sigma}, \]
for some small~$\sigma = \sigma(k)>0$. It follows that
\[
\CS_1(M) \ll q^{-\sigma}M \big(\sum_{p\leq y} \frac{(q, p^k)^\sigma}{p}\Big) \ \CS_2(M),
\]
where
\[ \CS_2(M) := \ssum{d|q^\infty\\d\leq \CQ}\ssum{x/(2M\CQ)<n_1\leq n_2\leq x/(Md)\\(n_i, q)=1, \ P^+(n_i)\leq y} \frac{(q, d^k(n_2^k-n_1^k))^\sigma}{(1+|\delta|(Md)^k(n_2^k-n_1^k))^\sigma}. \]
Since
\[ \sum_{p\leq y} \frac{(q, p^k)^\sigma}{p} \ll \log\log y + \omega(q) \ll_{\ee} M^\ee, \]
we have
\begin{equation}\label{eq:rel-S1-S2}
\CS_1(M) \ll q^{-\sigma} M^{1+\ee} \CS_2(M). 
\end{equation}
To bound~$\CS_2(M)$, splitting according to the value of~$r = (q, d^k(n_2^k-n_1^k))$, we obtain
\begin{equation}
\CS_2(M) \leq \sum_{r|q}r^\sigma \ssum{d|q^\infty\\d\leq \CQ}\CS_3(M; r, d),\label{eq:rel-S2-S3}
\end{equation}
where
\[ \CS_3(M; r,d) := \ssum{x/(2M\CQ)<n_1\leq n_2\leq x/(Md)\\(n_i, q)=1, \ P^+(n_i)\leq y\\r| d^k(n_2^k-n_1^k)} \frac{1}{(1+|\delta|(Md)^k(n_2^k-n_1^k))^\sigma}. \]
Note that~$r|d^k(n_2^k-n_1^k)$ is equivalent to~$n_1^k \equiv n_2^k\mod{r'}$, where~$r' := r/(r, d^k)$. Since~$(n_i, q)=1$, and since there are~$O((r')^\ee)$ residue classes~$b\mod{r'}$ such that~$(b, r')=1$ and~$b^k\equiv 1 \mod{r'}$, we deduce
\begin{equation}
\CS_3(M; r,d) \ll_{\ee} r^\ee \ssum{x/(2M\CQ)< n_1\leq x/(Md)\\P^+(n_1)\leq y} \sup_{\substack{b\mod{r'}\\(b, r')=1}} \CS_4(M; r',d; n_1,b) \label{eq:rel-S3-S4}
\end{equation}
for any~$\ee > 0$, where
\[ \CS_4(M; r',d; n_1, b) := \ssum{n_1\leq n_2\leq x/(Md)\\P^+(n_2)\leq y\\n_2\equiv b\mod{r'}} \frac{1}{(1+|\delta| (Md)^k (n_2^k-n_1^k))^\sigma}. \]
We dyadically decompose this sum with respect to the size of~$n_2-n_1 \in [0, x/(Md)]$, noting that if~$ T/2 \leq n_2-n_1 \leq T$, then~$n_2^k-n_1^k \geq (n_2-n_1)n_1^{k-1} \gg Tn_1^{k-1}$. Therefore,
\begin{equation}
\CS_4(M; r',d; n_1, b) \ll (\log x)\sup_{1\leq T\leq x/(Md)} \frac{\CS_5(M; r',d'; n_1,b'; T)}{(1+|\delta|(Md)^k Tn_1^{k-1})^\sigma}, \label{eq:majo-cS4}
\end{equation}
where
\[ \CS_5(M; r',d; n_1,b; T) := |\{n_2\in \Psi(x/(Md), y) :\ |n_2-n_1|\leq T,\ n_2\equiv b\mod{r'}\}|. \]
An application of Lemma~\ref{lem:adam-smooth-3} yields
\[
\CS_5(M; r',d; n_1,b; T) \ll \big(\frac{T/r'}{x/Md}\big)^{\alpha}\Psi(x/(Md), y) \log x + 1,
\]
where we have used~$\alpha(x/Md, y)\geq \alpha(x, y)$. Combining this with \eqref{eq:majo-cS4} and noting that the bound is an increasing function of~$T$ assuming~$\sigma < \alpha$ (which we may), we obtain
\[ \CS_4(M; r',d; n_1,b) \ll (\log x)^2\Big\{\frac{\Psi(x/(Md), y)}{(r')^\alpha (1+|\delta|x(Mdn_1)^{k-1})^\sigma} + 1\Big\}. \]
Inserting this into~\eqref{eq:rel-S3-S4} and recalling~$r' = r/(r,d^k)$, we obtain
\[ \CS_3(M; r,d) \ll \frac{(\log x)^2(r, d^k)^{\alpha}}{r^{\alpha-\ee}}\Big\{ \CS_3'(M; d) + \Psi(x/(Md), y) \Big \}, \]
where
\[ \CS_3'(M; d) := \Psi(x/(Md), y) \ssum{x/(2M\CQ)< n_1 \leq x/(Md)\\P^+(n_1) \leq y} \frac{1}{(1+|\delta|x(Mdn_1)^{k-1})^{\sigma}} \ll \frac{\Psi(x/(Md), y)^2}{(1+|\delta|x^k)^\sigma} \]
by partial summation assuming~$\sigma<\alpha/k$ (which we may). Since~$Md \leq yL\CQ \leq x/\CQ$ by our hypothesis, we have by Lemma~\ref{lem:adam-smooth-2},
$$ \Psi(x/(Md), y) \gg \big(\frac{x}{Md}\big)^\alpha \gg \CQ^{\alpha} \gg (1+|\delta|x^k)^\sigma, $$
and thus
\[ \CS_3(M; r,d) \ll \frac{(\log x)^2(r, d^k)^{\alpha}}{r^{\alpha-\ee}} \cdot \frac{\Psi(x/(Md), y)^2}{(1+|\delta|x^k)^{\sigma}} \ll \frac{(\log x)^2(r, d^k)^{\alpha}}{r^{\alpha-\ee}d^{2\alpha}}\frac{\Psi(x/M, y)^2}{(1+|\delta|x^k)^{\sigma}} \]
again by Lemma~\ref{lem:adam-smooth-2}. Inserting this bound into~\eqref{eq:rel-S2-S3}, we obtain
\[ \CS_2(M) \ll (\log x)^2 \frac{\Psi(x/M, y)^2}{(1+|\delta|x^k)^{\sigma}} \sum_{r \mid q} \ssum{d \mid q^{\infty} \\ d \leq \CQ}  \frac{(r, d^k)^\alpha}{r^{\alpha-\sigma-\ee}d^{2\alpha}}. \]
Writing~$r' = (r,d^k)$, the double sum over~$r$ and~$d$ above can be bounded by
\[  q^\ee \sum_{r|q}r^{\sigma-\alpha}\sum_{r'|r}(r')^\alpha \ssum{d|q^\infty\\r'|d^k}d^{-2\alpha}. \]
The inner sum over~$d$ is less than
\[ \big(\min\{r'' : r'|(r'')^k\}\big)^{-2\alpha} \sum_{d|q^\infty}d^{-2\alpha} \ll (r')^{-2\alpha/k}, \]
so that
\[ \sum_{r|q}\ssum{d|q^\infty\\d\leq D} \frac{(r, d^k)^\alpha}{r^{\alpha-\sigma-\ee}d^{2\alpha}} \ll
q^\ee \sum_{r|q}r^{\sigma-\alpha}\sum_{r'|r}(r')^{\alpha(1-2/k)} \ll q^\ee \sum_{r|q}r^{\sigma-2\alpha/k} \ll q^{2\ee}. \]
It follows that
\[ \CS_2(M) \ll_{\ee} q^\ee (\log x)^2 \frac{\Psi(x/M, y)^2}{(1+|\delta|x^k)^{\sigma}} \ll \frac{q^\ee(\log x)^2 M^{-2\alpha}\Psi(x, y)^2}{(1+|\delta|x^k)^{\sigma}} \]
for any~$\ee > 0$.  Finally, inserting this into~\eqref{eq:rel-S1-S2} we obtain 
\[ \CS_1(M) \ll_{\ee} \frac{(\log x)^2 M^{1-2\alpha+\ee}\Psi(x, y)^2}{(q(1+|\delta|x^k))^{\sigma - \ee}}, \]
and thus by~\eqref{eq:etape-CS} we have
\[ \CE(M) \ll_{\ee} (\log x)^2 M^{1-\alpha+\ee}\Psi(x, y)\CQ^{-\sigma/2+\ee} \]
for any~$\ee > 0$. The desired bound follows from a dyadic summation over~$M$, since~$M^{1-\alpha} \leq (yL)^{1-\alpha} \ll (y^2\CQ^2)^{1-\alpha} \ll (\log x)^2 \CQ^{2(1-\alpha)}$.
\end{proof}

\subsection{Deduction of Theorem~\ref{thm:estim-minor}}

We now have all the ingredients to deduce Proposition~\ref{prop:estim-minor} (and thus Theorem~\ref{thm:estim-minor}). Let the situation be as in the statement of Proposition~\ref{prop:estim-minor}. Let~$\eta > 0$ be a sufficiently small constant. If~$y \geq x^{\eta}$, then the conclusion follows from Proposition~\ref{prop:weyl-large-y}. Now assume that~$y \leq x^{\eta}$. If~$\vth \notin \major(x^{0.1}, x)$, then Proposition~\ref{prop:weyl-smooth-wooley} applies with~$\lambda = 0.1$ to give the desired conclusion. Finally, assume that~$y \leq x^{\eta}$ and~$\vth \in \major(x^{0.1}, x)$. Then~$\vth = a/q + \delta$ for some~$0 \leq a \leq q \leq x^{0.1}$ with~$(a,q) = 1$ and~$|\delta| \leq q^{-1}x^{-k+0.1}$. Thus~$\CQ := q(1 + |\delta|x^k) \leq 2x^{0.1}$, and the hypothesis of Proposition~\ref{estim-verysmooth-minor} is satisfied.  Moreover, the assumption~$\vth \notin \major(Q, x)$ implies that~$\CQ \geq Q$, and thus the conclusion of Proposition~\ref{estim-verysmooth-minor} implies that
\[ E_k(x,y; \vth) \ll \Psi(x,y) Q^{-c} (\log x)^5 \]
for some constant~$c > 0$, when~$1-\alpha$ is sufficiently small. This gives the desired bound when~$Q$ is at least a large power of~$\log x$. If~$Q \leq (\log x)^A$ for some constant~$A$, then Theorem~\ref{thm:estim-majo} applies and the conclusion follows from \eqref{eq:arcmaj-majoration}.

\section{Mean value estimates: statements of results}\label{sec:mean}

The goal of this section and the next is to prove Theorem~\ref{cor:mv-k}. In this section, we reduce the task of proving Theorem~\ref{cor:mv-k} to proving Proposition~\ref{prop:mv} below that controls large values of friable exponential sums. We start with the following mean value estimate, which holds with the optimal exponent when restricted to (relatively wide) major arcs.

\begin{prop}\label{thm:restriction}
Fix a positive integer $k$. The following statement holds for some sufficiently small~$c = c(k) > 0$. Let~$2 \leq y \leq x$ be large. Let~$(a_n)_{1 \leq n \leq x}$ be an arbitrary sequence of complex numbers, and write~$f(\vth)$ for the normalized exponential sum
\[ f(\vth) = \bigg( \sum_{n\in S(x,y)} |a_n|^2 \bigg)^{-1/2} \sum_{n\in S(x,y)} a_n \e(n^k\vth). \]
Then for any~$s>k$ we have
\[ \int_{\major} | f(\vth) |^{2s} \dd\vth \ll_s \Psi(x,y)^{s} x^{-k}, \]
where
\[ \major = \bigg\{ \vth \in [0,1]: |f(\vth)|^2 \geq x^{-c} \Psi(x,y) \bigg\}, \]
provided that~$1 - \alpha(x,y) \leq c \min(1, s-k)$.
\end{prop}

Proposition~\ref{thm:restriction} is a straightforward consequence of the following result, controlling the number of (well spaced) phases with large values of exponential sums.

\begin{prop}\label{prop:mv}
Fix a positive integer~$k$.
Let~$2 \leq y \leq x$ be large and let~$\alpha = \alpha(x,y)$. Let~$\{a_n\}_{1\leq n\leq x}$ be an arbitrary sequence of complex numbers, and write~$f(\vth)$ for the normalized exponential sum
\[ f(\vth) = \bigg( \sum_{n\in S(x,y)} |a_n|^2 \bigg)^{-1/2} \sum_{n\in S(x,y)} a_n \e(n^k\vth). \]
Let~$\vth_1,\cdots,\vth_R\in [0,1]$ be reals satisfying~$\|\vth_r-\vth_s\|\geq x^{-k}$ for any~$r\neq s$. Suppose that 
\[ |f(\vth)|^2 \geq \gamma^2 \Psi(x,y)  \]
for each~$1\leq r\leq R$ and some~$\gamma\in (0,1]$. If~$\gamma \geq x^{-c}$ and~$1 - \alpha \leq c$ for some sufficiently small~$c = c(k) > 0$, then~$R\ll_{\epsilon}\gamma^{-2k-O(1-\alpha)-\ee}$ for any~$\varepsilon > 0$.
\end{prop}

Large value estimates for complete Weyl sums of this type first appeared in \cite{Bou89}. For friable exponential sums with~$k=1$, this is proved by Harper \cite{Harper}.

In the remainder of this section, we give the standard deduction of Proposition~\ref{thm:restriction} from Proposition~\ref{prop:mv}, and also deduce Theorem~\ref{cor:mv-k} from Proposition~\ref{thm:restriction}. The proof of Proposition~\ref{prop:mv} is the topic of Section~\ref{sec:proof-large}.

\subsection{Proof of Proposition~\ref{thm:restriction} assuming Proposition~\ref{prop:mv}}

Note the trivial bound~$|f(\vth)|^2 \leq \Psi(x,y)$ which follows from the Cauchy-Schwarz inequality. For any~$\gamma \in (0,1]$, define
\[ S(\gamma) = \{ \vth \in [0,1]: |f(\vth)|^2 \geq \gamma^2 \Psi(x,y) \}. \]
Let~$c > 0$ be sufficiently small. We claim that if~$\gamma \in (x^{-c},1]$, then
\[
\measure(S(\gamma)) \ll_{\varepsilon} \gamma^{-2k- O(1-\alpha) - \varepsilon}x^{-k}, 
\]
for any~$\ee > 0$. To prove this claim, pick a maximal~$x^{-k}$ separated set of points~$\{\vth_1,\cdots,\vth_R\} \subset S(\gamma)$. In other words, the set~$\{\vth_1,\cdots,\vth_R\}$ satisfies~$\|\vth_r-\vth_s\| \geq x^{-k}$ for any~$r\neq s$, and moreover for any~$\vth \in S(\gamma)$ we have~$\|\vth - \vth_r\| \leq x^{-k}$ for some~$r$. Hence~$S(\gamma)$ is contained in the union of arcs centered around~$\vth_r$ ($1 \leq r\leq R$) with length~$2x^{-k}$, and the claim follows from Proposition~\ref{prop:mv}. By the assumption on~$1- \alpha$, we may ensure that
\[ \measure(S(\gamma)) \ll \gamma^{-s-k} x^{-k}. \]
Now write
\begin{align*} 
\int_{S(x^{-c})} |f(\vth)|^{2s} \dd\vth &=  2s \Psi(x,y)^s \int_0^1 \int_0^1 \gamma^{2s-1} \1_{\vth \in S(\gamma) \cap S(x^{-c})} \dd\gamma \dd\vth \\
&= \Psi(x,y)^s \left( 2s \int_{x^{-c}}^1 \gamma^{2s-1} \measure(S(\gamma)) \dd\gamma + O\left( x^{-2cs} \measure(S(x^{-c})) \right) \right). 
\end{align*}
The conclusion follows since
\[ \int_{x^{-c}}^1 \gamma^{2s-1} \measure(S(\gamma)) \dd\gamma \ll_s x^{-k} \int_{x^{-c}}^1 \gamma^{s-k-1} \dd\gamma \ll_s x^{-k} \]
and
\[  x^{-2cs} \measure(S(x^{-c})) \ll x^{-c(s-k)} x^{-k} \ll x^{-k}. \]

%
%

\subsection{Proof of Theorem~\ref{cor:mv-k} assuming Proposition~\ref{thm:restriction}}


In view of Proposition~\ref{thm:restriction}, Theorem~\ref{cor:mv-k} follows from Lemma~\ref{lem:mv-k} below.

\begin{lemma}\label{lem:mv-k}
Fix a positive integer~$k$. There exists~$p=p(k) \geq 2k$ such that
\[
\int_0^1 | E_k(x,y;\vth) |^{p} \dd\vth \ll_{p,\ee} x^{p-k+\ee} 
\]
for any~$\ee>0$. Moreover, we may take~$p(1)=2$,~$p(2)=4$, and~$p(3)=8$. If~$y \leq x^c$ for some sufficiently small~$c = c(k) > 0$, then we may take~$p(3)=7.5907$ and~$p(k)=k(\log k + \log\log k + 2 + O(\log\log k/\log k))$ for large~$k$.
\end{lemma}

Indeed, to deduce Theorem~\ref{cor:mv-k} from this lemma, let~$c > 0$ be sufficiently small and denote by~$\Fm$ the set of~$\vth \in [0,1]$ with
\[ | E_k(x,y;\vth) | \leq x^{-c} \Psi(x,y). \]
The contribution to the mean value integral from those~$\vth \notin \Fm$ is dealt with by Proposition~\ref{thm:restriction}. Thus it suffices to show that
\[ \int_{\Fm} | E_k(x,y;\vth) |^{2s} \dd\vth \ll \Psi(x,y)^{2s} x^{-k} \]
whenever~$2s > p$, where~$p = p(k)$ is the exponent in Lemma~\ref{lem:mv-k}. To prove this, bound the left hand side by
\[ (x^{-c} \Psi(x,y))^{2s-p} \int_{\Fm} |E_k(x,y;\vth)|^{p} \dd\vth \ll_{p,\varepsilon} x^{-c(2s-p) + p-k + \varepsilon} \Psi(x,y)^{2s-p} \]
using Lemma~\ref{lem:mv-k}. This bound is~$O( \Psi(x,y)^{2s} x^{-k})$ if~$1-\alpha \leq [c(2s-p) - \varepsilon]/p$. The conclusion follows if we choose~$\varepsilon = c(2s-p)/2$.

\begin{proof}[Proof of Lemma~\ref{lem:mv-k}]
First note that for~$p(k)=2^k$ we have
\[ \int_0^1 \left| E_k(x,y; \vth) \right|^{2^k} \dd\vth  \leq \int_0^1 |E_k(x,x;\vth)|^{2^k} \dd\vth \]
by considering the underlying diophantine equation. The right side above is bounded by~$x^{2^k-k+\ee}$ for any~$\ee>0$ by Hua's lemma (see~\cite[Lemma 2.5]{Vaughan97}). This shows the existence of~$p(k)$ as well as the choice of~$p(k)$ for~$k\in\{1,2,3\}$.

Now assume that~$y \leq x^c$ for some sufficiently small~$c = c(k) > 0$.  The fact that we may take~$p(3) = 7.5907$ follows from~\cite[Theorem 1.4]{WooleyThreeCubesII} or~\cite[formula~(6.3)]{WooleyThreeCubesII}. For large~$k$, the claimed choice for~$p(k)$ follows from Wooley's work on Waring's problem and friable Weyl sums \cite{Woo92,Woo95}, together with arguments very close to those in~\cite[Section 5]{Vau89} that deal with major arcs. For completeness, we include the details here.


Let~$k$ be large and let~$p = k(\log k + \log\log k + 2 + C\log\log k/\log k)$ be an even integer for some large constant~$C > 0$. 
By considering the underlying diophantine equation, we obtain
\[ \int_0^1 \left| E_k(x, y; \vth) \right|^{p} \dd\vth   \leq \int_0^1 |E_k(x,x^c;\vth)|^{p-2} |E_k(x,x;\vth)|^2 \dd\vth. \]
The two copies of the complete exponential sum are required in the major arc analysis. Call the right hand side above~$T$, and our goal is to show that~$T\ll x^{p-k}$.
For~$0 \leq a \leq q\leq x$ and~$(a,q)=1$, define 
\[ \major(q,a) = \left\{ \vth\in [0,1]: | q\vth-a| \leq 1/(2kx^{k-1}) \right\}, \]
and let~$\major$ be the union of all these.  
Split~$T$ into two integrals
\[ T_1 = \int_{\major} |E_k(x,x^c;\vth)|^{p-2} |E_k(x,x;\vth)|^2 \dd\vth\]
and
\[ T_2 = \int_{[0,1]\setminus\major} |E_k(x,x^c;\vth)|^{p-2} |E_k(x,x;\vth)|^2 \dd\vth. \]
To bound~$T_1$, by H\"{o}lder's inequality we have
\[ T_1 \leq  \left( \int_0^1 |E_k(x,x^c;\vth)|^p \dd\vth \right)^{(p-2)/p} \left( \int_{\major} |E_k(x,x;\vth)|^p \dd\vth \right)   ^{2/p}. \]
The first integral above is at most~$T$ by considering the underlying diophantine equation and the second integral over~$\major$ can be bounded by~$x^{p-k}$ (see~\cite[Lemma 5.1]{Vau89}). Hence
\[ T_1\ll  T^{(p-2)/p} x^{2(p-k)/p}. \]
To bound~$T_2$, we use the trivial bound~$|E_k(x,x;\vth)|\leq x$ and take out~$t$ copies of the minor arc exponential sum, where~$t\in\{k,k+1\}$ is even:
\[ T_2 \leq x^2 \big( \sup_{\vth\notin\major} |E_k(x,x^c;\vth)| \big)^t \int_0^1 |E_k(x,x^c;\vth)|^{p-2-t} \dd\vth. \]
From~\cite[Theorem 1.1]{Woo95} we have
\[ \sup_{\vth\notin\major} |E_k(x,x^c;\vth)| \ll_{\ee} x^{1-\rho(k)+\ee} \]
for any~$\ee > 0$, provided that~$c$ is sufficiently small depending on~$\ee$. Here~$\rho(k)>0$ satisfied~$\rho(k)^{-1} = k (\log k + O(\log\log k))$. From~\cite[Lemma 2.1]{Woo95}, for any positive integer~$s$ we have
\[  \int_0^1 |E_k(x,x^c;\vth)|^{2s} \dd\vth \ll_{\ee} x^{2s-k+\Delta_{s,k}+\ee} \]
for any~$\ee > 0$, where~$\Delta_{s,k}=ke^{1-2s/k}$. Apply this with~$2s=p-2-t$ to get
\[ T_2 \ll x^{p-k+\ee} x^{\Delta_{s,k}-\rho(k)t}   \]
for any~$\ee > 0$. Since~$2s=p-2-t \geq k(\log k + \log\log k + 1 + (C-1)\log\log k/\log k)$ for large~$k$, we have
\[ \Delta_{s,k}\leq \frac{1}{\log k} \exp\left(-(C-1)\frac{\log\log k}{\log k}\right) \leq \frac{1}{\log k}\left(1 - \frac{C}{2} \cdot \frac{\log\log k}{\log k} \right) . \]
This implies that 
\[ \rho(k)t - \Delta_{s,k} \geq \rho(k)k - \Delta_{s,k} \geq \frac{C}{4} \cdot \frac{\log\log k}{(\log k)^2}, \]
and thus~$T_2\ll x^{p-k}$. Combining the bounds for~$T_1$ and~$T_2$ we obtain
\[ T \ll T^{(p-2)/p} x^{2(p-k)/p} + x^{p-k}. \]
This implies the desired bound~$T\ll x^{p-k}$.
\end{proof}

%

\section{Proof of the large value estimates}\label{sec:proof-large}

The goal of this section is to prove Proposition~\ref{prop:mv}. Let~$c > 0$ be a sufficiently small constant. We may clearly assume that~$\varepsilon \leq c$. We may also assume that~$y \leq x^c$, since otherwise~$\Psi(x,y) \gg x$ and the conclusion follows from Bourgain's work~\cite[Section 4]{Bou89}. Recall also that we are able to assume~$1-\alpha \leq c$ and~$\gamma \geq x^{-c}$.

Using the major arc estimates in Theorem~\ref{thm:estim-majo}, Bourgain's argument~\cite{Bou89} can be followed to treat the case when~$\gamma^{-1}$ is smaller than a fixed power of~$\log x$. When~$\gamma^{-1}$ is larger, we will use well factorability of friable numbers to arrive at a double sum, and after applying the Cauchy-Schwarz inequality we will be able to drop the friability restriction on one of the sums, in order to take advantage of good major arc estimates for complete exponential sums. 

We now turn to the details. For each~$1\leq r\leq R$, let~$\eta_r$ be a complex number with~$|\eta_r|=1$ such that~$|f(\vth_r)| = \eta_r f(\vth_r)$. From the assumption that
\[ |f(\vth_r)|^2 \geq \gamma^2 \Psi(x,y) \]
for each~$1 \leq r \leq R$, we obtain
\[ \sum_{1\leq r\leq R}\eta_r\sum_{n \in S(x,y)} a_n \e(n^k\vth_r)\geq \gamma R\Psi(x,y)^{1/2}  \bigg( \sum_{n \in S(x,y)} |a_n|^2 \bigg)^{1/2}. \]
An application of the Cauchy-Schwarz inequality after changing the order of summation in~$r$ and~$n$ leads to
\begin{equation}\label{eq:mean0} 
\sum_{n \in S(x,y)} \left|\sum_{1\leq r\leq R} \eta_r \e(n^k\vth_r)\right|^2\geq \gamma^2R^2\Psi(x,y). 
\end{equation}

\subsection{The case of large~$\gamma$.}\label{sec:proof-large-3}

Let us first assume that~$\gamma^{-1}\leq \min((\log x)^B, y^c)$ for some large constant~$B = B(k,\ee)$. In this subsection, we allow all implied constants to depend on~$B$. Expand the square in~\eqref{eq:mean0} to find
\begin{equation}\label{eq:large-gamma-rs}
 \sum_{1\leq r,s\leq R}\bigg|\sum_{n \in S(x,y)} \e(n^k(\vth_r-\vth_s)) \bigg|\geq \gamma^2R^2\Psi(x,y). 
\end{equation}
Let~$\FQ$ be the set of~$\vth\in [0,1]$ with~$|E_k(x,y; \vth)| \geq \gamma^2\Psi(x,y)/2$.
Then
\begin{equation}\label{eq:Ethetars} 
\ssum{1\leq r,s\leq R\\ \vth_r-\vth_s\in\FQ} |E_k(x,y; \vth_r-\vth_s)| \geq \frac{1}{2}\gamma^2 R^2\Psi(x,y). 
\end{equation}

\begin{lemma}\label{lem:large-gamma}
Let the notations and assumptions be as above (in particular, assume~$\gamma^{-1}\leq (\log x)^B$). If~$\vth\in\FQ$, then~$\vth=a/q+\delta$ for some~$(a,q)=1$ with~$\CQ = q(1+|\delta| x^k) \ll \gamma^{-3k}$. Moreover, we have
\[ |E_k(x,y; \vth)| \ll_{\ee, B} \Psi(x,y) \CQ^{-1/k+2(1-\alpha)+\ee}, \]
for any~$\varepsilon>0$.
\end{lemma}

\begin{proof}
Since~$E_k(x,y; \vth)\geq \gamma^2 \Psi(x, y)/2$, Proposition~\ref{prop:estim-minor} implies that~$\vth\in\major(\gamma^{-C}, x)$ for some~$C=C(k)>0$. Since~$\gamma^{-1} \leq \min((\log x)^B, y^c)$, we may apply Theorem~\ref{thm:estim-majo} (in particular the estimate~\eqref{eq:arcmaj-majoration}) to obtain the desired upper bound for~$E_k(x,y; \vth)$. Combining this upper bound with the lower bound~$E_k(x,y; \vth)\geq \gamma^2 \Psi(x, y)/2$, we get~$\CQ \ll \gamma^{-3k}$ as desired.

\end{proof}

We are now in a position to apply Lemma~\ref{lem:restriction}. Let~$Q = C\gamma^{-3k}$ for some large constant~$C > 0$, and let~$\Delta = Qx^{-k}$. Consider the function~$G = G_{x^k, Q, \Delta}$ defined by
\[ G(\vth) = \sum_{q \leq Q} \frac{1}{q} \sum_{a=0}^{q-1} \frac{ \1_{\|\vth - a/q\| \leq \Delta} }{1 + x^k\|\vth - a/q\|}. \]
Lemma~\ref{lem:large-gamma} implies that
\[ E_k(x,y; \vth) \ll \Psi(x,y) G(\vth)^{1/k} \gamma^{-6k(1-\alpha+\varepsilon)} \]
whenever~$\vth \in \FQ$. Comparing this with~\eqref{eq:Ethetars} we obtain
\[ \gamma^2 R^2 \Psi(x,y) \ll \Psi(x,y) \gamma^{-6k(1-\alpha+\varepsilon)} \sum_{1 \leq r,s\leq R} G(\vth_r - \vth_s)^{1/k}, \]
which simplifies to
\[ \sum_{1 \leq r,s \leq R} G(\vth_r - \vth_s)^{1/k} \gg R^2 \gamma^{2 + 6k(1-\alpha+\varepsilon)}. \]
On the other hand, by H\"{o}lder's inequality and Lemma~\ref{lem:restriction} we have
\begin{align*}
\sum_{1\leq r,s\leq R} G(\vth_r-\vth_s)^{1/k} &\leq R^{2(k-1)/k} \bigg(\sum_{1\leq r,s\leq R} G(\vth_r-\vth_s) \bigg)^{1/k} \\
& \ll R^{2(k-1)/k} \left[ (R\gamma^{-\epsilon} + x^{-k}R^2\gamma^{-3k} + \gamma^{A} R^2) \log (1+\gamma^{-3k}) \right]^{1/k},
\end{align*}
for any~$A > 0$. Combining this with the lower bound we arrive at
\[ R^2 \gamma^{2 + 6k(1-\alpha+2\varepsilon)} \ll R^{2-1/k} + R^2 x^{-1}\gamma^{-3} + R^2 \gamma^A \]
for any~$A > 0$. The second and the third terms on the right above are clearly smaller than the left hand side. Hence
\[ R^2 \gamma^{2 + 6k(1-\alpha+2\varepsilon)} \ll R^{2-1/k}. \]
This leads to the desired upper bound on~$R$.

\subsection{The case of small~$\gamma$}\label{sec:proof-small-gamma}

In the remainder of this section, we will assume that~$\gamma^{-1}\geq \min((\log x)^B, y^c)$ for some large enough~$B = B(k,\ee) > 0$. In particular, this implies that either~$\gamma^{-1} \geq (\log x)^B$ or~$\gamma^{-(1-\alpha)} \geq (\log x)^{c/2}$. Let~$K=(\gamma^{-1}\log x)^{A}$ be a parameter, where~$A = A(k) > 0$ is a large constant to be specified later. By the assumption~$\gamma \geq x^{-c}$ we may assume that~$K \leq x^{1/2k}$.  Observe that any integer in~$S(x, y)$ can be written as a product~$mn$, where~$m\in [x(yK)^{-1},xK^{-1}]$ is~$y$-friable, and~$n\leq xm^{-1}$. In this way we get from~\eqref{eq:mean0}
\[ \sum_{\substack{x(yK)^{-1}\leq m\leq xK^{-1}\\ P^+(m)\leq y}} \sum_{1 \leq n \leq xm^{-1}} \left|\sum_{1 \leq r \leq R} \eta_r \e(n^km^k\vth_r)\right|^2\geq \gamma^2R^2\Psi(x,y). \]
Expand the square and move the sum over~$n$ inside to get
\begin{equation}\label{eq:mean1} 
\sum_{\substack{x(yK)^{-1}\leq m\leq xK^{-1}\\ P^+(m)\leq y}} \sum_{1\leq r,s\leq R}\bigg|\sum_{1\leq n\leq xm^{-1}} \e(n^km^k(\vth_r-\vth_s))\bigg|\geq \gamma^2R^2\Psi(x,y). 
\end{equation}
This is similar as~\eqref{eq:large-gamma-rs} in Section~\ref{sec:proof-large-3}, but we have arranged the inner sum to be a complete Weyl sum, at some cost since the trivial bound for the left hand side is now larger. The assumption~$\gamma^{-1}\geq \min((\log x)^B, y^c)$ will ultimately ensure that this cost is acceptable.

It is convenient to perform a dyadic division in~$m$. For each~$M\in [x(yK)^{-1},xK^{-1}]$ and~$\vth\in\R$, define
\begin{equation}\label{eq:IM-defn} 
I_M(\vth) = \ssum{M \leq m \leq 2M\\ P^+(m)\leq y} \bigg|\sum_{1\leq n\leq xm^{-1}} \e(n^km^k\vth)\bigg|,
\end{equation}
and
\begin{equation}\label{eq:IM-defn-2}
I_M = \sum_{1 \leq r,s \leq R} I_M(\vth_r - \vth_s). 
\end{equation}
For ease of notation we write~$N=xM^{-1}$ so that~$N\in [K,yK]$. We will show in Sections~\ref{sec:proof-large-1} and~\ref{sec:proof-large-2}, that for all fixed~$\ee>0$,
\begin{equation}\label{eq:IM-bound} 
I_M \ll_{\varepsilon} R^2 N \Psi(2M,y) (R^{-1/k} + K^{-c}) K^{1-\alpha+\varepsilon} (\log x). 
\end{equation}
Let us temporarily assume~\eqref{eq:IM-bound} and deduce the conclusion of Proposition~\ref{prop:mv}. Note that~$\Psi(2M, y) \ll N^{-\alpha} \Psi(x,y)$ from Lemma~\ref{lem:adam-smooth-2}. We may combine~\eqref{eq:IM-bound} with~\eqref{eq:mean1} and obtain, after summing over~$M$ (or~$N$) dyadically, that
\[  \gamma^2 R^2 \Psi(x,y) \ll R^2 \Psi(x,y) (R^{-1/k} + K^{-c}) K^{2(1-\alpha) + \varepsilon} (\log x)^3, \]
where we used the following estimate for the dyadic sum:
\[ \sum_{0 \leq j \leq \lceil \log_2 y\rceil} (2^j K)^{1-\alpha} \ll \frac{(yK)^{1-\alpha}}{2^{1-\alpha}-1} \ll K^{1-\alpha}(\log x)^2. \]
This simplifies to
\[ \gamma^2 \ll (R^{-1/k} + K^{-c}) K^{2(1-\alpha)+\varepsilon} (\log x)^3. \]
If the second term on the right hand side dominates, then
\[ \gamma^2 \ll K^{-c + 2(1-\alpha) + \varepsilon} (\log x)^3 \ll K^{-c/2} (\log x)^3, \]
and thus~$K \ll (\gamma^{-1} \log x)^{8/c}$, contradicting our choice of~$K$ if~$A$ is large enough. Thus we must have
\[ \gamma^2 \ll R^{-1/k} K^{2(1-\alpha) + \varepsilon} (\log x)^3. \]
After rearranging and recalling the choice of~$K$ we get
\[ R \ll \gamma^{-2k} K^{2k(1-\alpha) + k\varepsilon} (\log x)^{3k} = \gamma^{-2k - 2kA(1-\alpha) - kA\varepsilon} (\log x)^{3kA}. \]
Since either~$\gamma^{-1} \geq (\log x)^B$ or~$\gamma^{-(1-\alpha)} \geq (\log x)^{c/2}$, the~$(\log x)^{3kA}$ term can be absorbed so that
\[ R \ll  \gamma^{-2k - O(1-\alpha) - 2kA\varepsilon}. \]
The proof is completed after reinterpreting~$\varepsilon$ by~$\varepsilon/(10kA)$. We are therefore left to prove the bound~\eqref{eq:IM-bound}.

\subsection{Handling the minor arcs}\label{sec:proof-large-1}

Fix~$M \in [x(yK)^{-1}, xK^{-1}]$ and~$N = xM^{-1} \in [K, yK]$. In this section we prove that
\begin{equation}\label{eq:IM-theta-minor}
I_M(\vth) \ll N K^{-c}  \Psi(2M, y)
\end{equation} 
whenever~$\vth \in \Fn$, where the minor arc~$\Fn$ is the complement of~$\FN = \major(K^{1/2}, x)$ (recall the notation~\eqref{eq:def-majorarcs}). In particular, this means that those pairs~$(r,s)$ with~$\vth_r - \vth_s \in \Fn$ make an acceptable contribution in the sum~\eqref{eq:IM-defn-2} towards the bound in~\eqref{eq:IM-bound}.

For the rest of this subsection, fix some~$\vth \in \Fn$. We also need the auxiliary major arc~$\FQ = \major(K^{\eta},N)$ for some small~$\eta > 0$ to be specified later. Let~$\Fq$ be the complement of~$\FQ$. If~$m^k\vth \in \Fq$ for some~$m \in [M,2M]$, then by Weyl's inequality (Lemma~\ref{lemme:Weyl-qpetit})
\[ \bigg|\sum_{1\leq n\leq xm^{-1}} \e(n^km^k\vth)\bigg| \ll N K^{-\sigma\eta}\]
for some~$\sigma = \sigma(k) > 0$. Hence,
\[
I_M(\vth) =  \ssum{M\leq m\leq 2M\\ P^+(m)\leq y\\ m^k\vth\in\FQ} \bigg|\sum_{1\leq n\leq xm^{-1}} \e(n^km^k\vth)\bigg| + O(N K^{-\sigma\eta} \Psi(2M,y)).
\]
Bounding the inner sum over~$n$ above trivially by~$O(N)$, we reduce~\eqref{eq:IM-theta-minor} to proving the bound
\begin{equation}\label{eq:IM1}
\ssum{M\leq m\leq 2M\\ P^+(m)\leq y} \1_{m^k\vth \in \FQ}  \ll K^{-c} \Psi(2M,y).
\end{equation}
We will now divide into two cases, depending on whether or not~$\vth$ lies in the auxiliary major arcs~$\FP = \major(K^{1/5}, M)$ (which is wider than~$\FN$). Let~$\Fp$ be the complement of~$\FP$. We use the Erd\"{o}s-Tur\'{a}n inequality when~$\vth\in\Fp$, and use the combinatorial lemma, Lemma~\ref{lem:vinogradov}, when~$\vth \in \FP \cap \Fn$.

\begin{case}

First assume that~$\vth \in \Fp$. Since~$\FQ$ is the union of at most~$K^{2\eta}$ intervals of length at most~$2K^{\eta}N^{-k}$, the Erd\"{o}s-Turan inequality (Lemma~\ref{lem:erdos-turan}) gives
\begin{equation}\label{eq:et} 
\ssum{M \leq m\leq 2M\\ P^+(m)\leq y} \mathbf{1}_{ m^k\vth\in \FQ } \ll K^{2\eta}\bigg( \frac{K^{\eta}}{N^k}\Psi(2M,y) + \frac{\Psi(2M,y)}{J} + \sum_{j\leq J} \frac{1}{j} \bigg|\ssum{M \leq m\leq 2M \\ P^+(m) \leq y} \e(m^k j\vth) \bigg| \bigg),
\end{equation}
where~$J=K^{4\eta}$. The first two terms clearly make an acceptable contribution towards the bound in~\eqref{eq:IM1}. Thus it suffices to show that for each~$1\leq j\leq J$ we have
\begin{equation}\label{eq:minor} 
\bigg|\ssum{M \leq m\leq 2M\\ P^+(m)\leq y} \e(m^k j\vth) \bigg| \ll K^{-c} \Psi(2M,y),
\end{equation}
and then~\eqref{eq:IM1} follows if~$\eta$ is chosen small enough. Now fix~$j \leq J = K^{4\eta}$. Since~$\vth \notin \FP = \major(K^{1/5}, M)$, a moment's thought reveals that~$j \vth \notin \major(K^{1/5-4\eta}, M)$.  The desired bound~\eqref{eq:minor} then follows from Proposition~\ref{prop:estim-minor}.


\end{case}

\begin{case}
Now let~$\vth \in \FP = \major(K^{1/5}, M)$. We may choose~$0 \leq a \leq q \leq K^{1/5}$ with~$(a,q) = 1$, such that~$\vth \in \major(q,a; K^{1/5}, M)$.  Let~$A := \{m\in [M,2M], P^+(m)\leq y\}$, and assume that the proportion of elements~$m \in A$ satisfying~$m^k\vth \in \FQ = \major(K^{\eta}, N)$ is~$\delta$. Suppose for contradiction that~$\delta \geq K^{-c}$. We wish to show that this contradicts our hypothesis~$\vth \in \Fn$. 

If~$m \in A$ satisfies~$m^k \vth \in \FQ$, then~$\|m^k q_m \vth\| \leq K^{\eta}/N$ for some~$q_m \leq K^{\eta}$. By the pigeonhole principle, we may find~$q' \leq K^{\eta}$, such that the proportion of elements~$m \in A$ satisfying~$\|m^k q' \vth\| \leq K^{\eta}/N^k$ is at least~$\delta K^{-\eta}$. In particular, for these~$m$ we have
\[ \|m^k (q'q\vth)\| \leq K^{1/5+\eta}/N^k. \]
We will soon apply Lemma~\ref{lem:vinogradov} to the set~$A$ and the phase~$q'q\vth$, with~$\varepsilon = K^{1/5+\eta}/N^k$, but before that we need to figure out the permissible choices of the parameters~$L$ and~$\Delta$. Since
\[ \|q'q \vth\|  \leq K^{\eta} \|q\vth\| \leq K^{1/5+\eta}/M^k, \]
the condition~$\|q'q\vth\| \leq \varepsilon/(LM^{k-1})$ is satisfied with the choice~$L = M/N^k$. By Lemma~\ref{lem:adam-smooth-3}, for any arithmetic progression~$P \subset [M,2M] \cap \Z$ of length at least~$L$ we have
\[ |A \cap P| \ll |P|^{\alpha} \frac{\Psi(2M,y)}{M^{\alpha}}  \log M. \]
Thus we may choose~$\Delta$ with 
\[ \Delta \ll \left(\frac{M}{|P|}\right)^{1-\alpha} \log M \ll  N^{k(1-\alpha)} \log x \leq K^{k(1-\alpha)} (\log x)^{2k+1} \leq K^{1/4}, \]
where we used~$y^{1-\alpha} \ll (\log x)^2$ and~$(\log x)^{2k+1} \leq K^{1/8}$ if~$A$ (in the choice of~$K$) is large enough. The conclusion of Lemma~\ref{lem:adam-smooth-3} then says that either
\[ K^{1/5+\eta}/N^k \gg \delta K^{-1/4-\eta}, \]
or else
\[ \|q' q \vth\| \ll K^{1/4} (\delta K^{-\eta})^{-1} K^{1/5+\eta} /(MN)^k = \delta^{-1} K^{9/20+2\eta} x^{-k}. \]
The first case clearly implies that~$\delta \ll K^{-1/2}$, a contradiction. In the second case, since~$\delta^{-1} \leq K^c$ we have
\[ \|q'q \vth\| \leq K^{9/20+2\eta+c} x^{-k}. \]
Recalling~$q'q \leq K^{1/5+\eta}$, this implies~$\vth \in \FN$, giving the desired contradiction.

\end{case}

\subsection{Handling the major arcs}\label{sec:proof-large-2}

In view of~\eqref{eq:IM-theta-minor}, in order to prove~\eqref{eq:IM-bound} it suffices to show that
\begin{equation}\label{eq:mean2} 
\ssum{1\leq r,s\leq R\\ \vth_r-\vth_s\in\FN} I_M(\vth_r-\vth_s) \ll R^2 N \Psi(2M,y) (R^{-1/k} + K^{-1}) K^{1-\alpha+\varepsilon} (\log x).
\end{equation}
If $\vth \in \FN$ then~$m^k\vth$ also lies in appropriate major arcs so that the inner sum over~$n$ in the definition of~$I_M(\vth)$ in~\eqref{eq:IM-defn} can be controlled quite precisely. This analysis will lead to the following lemma (compare with Lemma~\ref{lem:large-gamma} above).

\begin{lemma}\label{lem:IM-major}
Let the notations be as above. Suppose that~$\vth\in\major(q,a;K^{1/2},x)$ for some~$0 \leq a \leq q\leq K^{1/2}$ and~$(a,q)=1$. Write~$\vth=a/q+\delta$ and let~$\CQ = q(1+|\delta| x^k)$. Then
\[
I_M(\vth) \ll_{\varepsilon}  N \Psi(2M,y)\CQ^{-1/k} q^{(1-\alpha)/k+\varepsilon}
\]
for any~$\varepsilon>0$.
\end{lemma}

\begin{proof}
Recall the definition of~$I_M(\vth)$ from~\eqref{eq:IM-defn}. Fix~$m \in [M,2M]$, and write~$q' = q/(q,m^k)$ and~$a'=am^k/(q,m^k)$. From standard major arc estimates for complete Weyl sums (see Lemma 2.8, Theorem 4.1, and Theorem 4.2 in~\cite{Vaughan97}), we have
\[ \sum_{1\leq n\leq xm^{-1}} \e(n^km^k\vth)=q'^{-1}S(q',a')v(\delta m^k)+O(\CQ^{1/2} q^{\varepsilon}), \]
where the (local) singular series~$S(q',a')$ and the (local) singular integral satisfy the bounds
\[ S(q',a')\ll q'^{1-1/k}, \ \ v(\beta)\ll \min(N,\|\beta\|^{-1/k}) \]
for~$|\beta| \leq 1/2$. It follows that
\[ \sum_{1\leq n\leq xm^{-1}} \e(n^km^k\vth)\ll  N \left( \frac{(q,m^k)}{\CQ} \right)^{1/k} +\CQ^{1/2} q^{\varepsilon} .  \]
Since~$\CQ \ll K^{1/2}$, the term~$\CQ^{1/2} q^{\varepsilon}$ clearly makes an acceptable contribution towards the desired bound for~$I_M(\vth)$. The first term contributes
$$ N \CQ^{-1/k} \ssum{M \leq m \leq 2M\\ P^+(m)\leq y} (q,m^k)^{1/k}. $$
The sum here is at most
$$ \sum_{d|q} d^{1/k} \ssum{M \leq m \leq 2M \\ P^+(m)\leq y \\ d | m^k} 1 \leq \sum_{d|q} d^{1/k}\Psi(2M/d^{1/k}, y) \ll q^{(1-\alpha)/k}\tau(q) \Psi(2M, y) $$
by using Lemma~\ref{lem:adam-smooth-2} and the inequality~$\alpha(2M, y) \geq \alpha(x, y)$. This completes the proof of the lemma.
\end{proof}

We are now in a position to apply Lemma~\ref{lem:restriction}. Let~$Q = K^{1/2}$ and~$\Delta = Qx^{-k}$. Consider the function~$G = G_{x^k, Q, \Delta}$ defined by
\[ G(\vth) = \sum_{q \leq Q} \frac{1}{q} \sum_{a=0}^{q-1} \frac{ \1_{\|\vth - a/q\| \leq \Delta} }{1 + x^k\|\vth - a/q\|}. \]
Lemma~\ref{lem:IM-major} implies that
\[ I_M(\vth) \ll N \Psi(2M,y) G(\vth)^{1/k} K^{1-\alpha+\varepsilon} \]
whenever~$\vth \in \FN$. Therefore,
\[ \ssum{1\leq r,s\leq R\\ \vth_r-\vth_s\in\FN} I_M(\vth_r-\vth_s) \ll N \Psi(2M,y) K^{1-\alpha+\varepsilon} \sum_{1 \leq r,s \leq R} G(\vth_r - \vth_s)^{1/k}. \]
To prove~\eqref{eq:mean2} it thus suffices to show that
\[ \sum_{1 \leq r,s \leq R} G(\vth_r - \vth_s)^{1/k} \ll R^{2} (R^{-1/k} + K^{-1}) K^{\varepsilon} (\log x) \]
for any~$\varepsilon > 0$. This is a straightforward consequence of H\"{o}lder's inequality and Lemma~\ref{lem:restriction}:
\begin{align*}
\sum_{1\leq r,s\leq R} G(\vth_r-\vth_s)^{1/k} &\leq R^{2(k-1)/k} \bigg(\sum_{1\leq r,s\leq R} G(\vth_r-\vth_s) \bigg)^{1/k} \\
& \ll R^{2(k-1)/k} \left[ (RK^{\epsilon} + x^{-k}R^2K^{1/2} + K^{-k} R^2) \log x \right]^{1/k},
\end{align*}
noting that the second term on the right hand side is dominated by the third term since~$x^{-k} K^{1/2} \leq K^{-k}$. This completes the proof of~\eqref{eq:mean2}, hence of~\eqref{eq:IM-bound}. By the arguments at the end of Section~\ref{sec:proof-small-gamma}, we have finished the proof of Proposition~\ref{prop:mv}.

\section{Waring's problem in friable variables}\label{sec:waring}

In this section we prove Theorem~\ref{thm:waringk}, getting an asymptotic formula for the number of representations of a large enough positive integer~$N$ as the sum of~$s$~$k$th powers of~$(\log N)^C$-friable numbers for some sufficiently large~$C$, as long as~$s$ exceeds a threshold depending on~$k$ which is essentially the same as that in the classical Waring's problem. 

Let notations and assumptions be as in the statement of Theorem~\ref{thm:waringk}. We start by defining the archimedian factor~$\beta_{\infty}$ and the local factors~$\beta_p$ that appear in the statement of Theorem~\ref{thm:waringk}.

\begin{defi}[The archimedian factor]
The archimedean factor~$\beta_{\infty}$ is defined by
\begin{equation}\label{eq:arch-factor} 
\beta_{\infty} = \int_{-\infty}^{+\infty} \cPhi(\delta,\alpha)^s e(-\delta) \dd\delta, 
\end{equation}
where~$\cPhi$ is defined in~\eqref{eq:def-cPhi}.
\end{defi}

We have the following explicit formula for~$\beta_\infty$, showing that~$\beta_{\infty} \asymp_s 1$ as long as~$\alpha$ is bounded away from~$0$.

\begin{prop}\label{prop:beta-infty}
The archimedian factor~$\beta_{\infty}$ defined above satisfies
\[ \beta_{\infty} =  \Gamma(s\alpha/k)^{-1} \Gamma(\alpha/k+1)^s. \]
\end{prop}
\begin{proof}
A change of variables~$t\gets t^{1/k}$ shows that~$\delta\mapsto\cPhi(\delta, \alpha)$ is the Fourier transform of~$\Phi_{\alpha}(t) := (\1_{0<t<1} )(\alpha/k)t^{\alpha/k-1}$. Fourier inversion then implies that~$\beta_\infty$ is the value of the convolution~$s$-th power~$(\Phi_{\alpha})^{\ast s}(1)$. This value is computed using \textit{e.g.}~\cite[Exercice~144]{ITAN} applied with~$n\gets s-1$ and~$f$ approaching~$u\mapsto (1-u)^{\alpha-1}$.
\end{proof}

To define the non-archimedian factors, we first define a probability measure~$\mu_q$ on~$\Z/q\Z$ for~$q=p^m$ a prime power, reflecting the bias that friable numbers are more likely to be divisible by a given small prime. For~$b \in \Z/q\Z$ with~$(b, p^m) = p^v$ for some~$0 \leq v \leq m$, we define
\[ \mu_{p^m}(b) = \begin{cases}
0 & v>0\text{ and }p>y, \\
\varphi(p^m)^{-1} & v=0\text{ and }p>y, \\
\varphi(p^m)^{-1}p^{(1-\alpha)v}(1-p^{-\alpha}) & v<m\text{ and }p\leq y, \\ 
p^{-\alpha m} & v=m\text{ and }p\leq y.
\end{cases} \]
Note that the value of~$\mu_{p^m}(b)$ depends only on~$v$. This is consistent with the heuristic model suggested by the approximation 
$$ \Psi(x/p^m, y) \approx p^{-m\alpha} \Psi(x, y), $$
(see~\cite[Th\'{e}or\`{e}me~2.4]{BT2005}).

\begin{defi}[The local factors]
For~$p$ prime, the local factor~$\beta_p$ is defined by
\begin{equation}\label{eq:local-factor} 
\beta_p = \lim_{m\rightarrow\infty} p^m \sum_{\substack{n_1,\cdots,n_s\mod{p^m}\\ n_1^k+\cdots+n_s^k\equiv N\mod{p^m}}} \mu_{p^m}(n_1) \cdots\mu_{p^m}(n_s)
\end{equation}
whenever the limit exists.
\end{defi}

Note that the sum above is the probability of the event~$n_1^k+\cdots+n_s^k\equiv N\mod{p^m}$ when~$n_1,\cdots,n_s$ are chosen according to the probability measure~$\mu_{p^m}$. When~$\alpha = 1$ and~$p \leq y$ this reduces to the uniform measure. In the appendix we will prove that the limit in~\eqref{eq:local-factor} does exist, and that the following estimates on the local factors hold.

\begin{prop}\label{prop:beta-p}
The local factors~$\beta_p$ are well-defined for every~$p$ and satisfy
\[ \prod_{p} \beta_p \asymp 1, \]
whenever~$\alpha > 2k/s$ and~$s \geq s_0(k)$ for some constant~$s_0(k)$. Moreover, we may take~$s_0(1)=3$, $s_0(2) = 5$, $s_0(3) = 5$, and~$s_0(k) =  O(k)$ for large~$k$.
\end{prop}

To prove Theorem~\ref{thm:waringk}, let~$Q=(\log x)^A$ for some sufficiently large constant~$A$.  Let~$\FM = \major(Q,x)$ (recall~\eqref{eq:def-majorarcs}), and let~$\Fm := [0, 1)\setminus \FM$ be its complement. By the circle method, the number of representations of~$N$ is 
\[ \int_0^1 E_k(x,y;\vth)^s \e(-N\vth) \dd\vth. \]
Theorem~\ref{thm:waringk} is easily seen to follow from the two lemmas below.

\begin{lemma}[Major arcs for Waring's problem]\label{lem:waring-major}
Let the notations and assumptions be as in the statement of Theorem~\ref{thm:waringk}, and let~$\FM$ be defined as above. Then
\[ \int_{\FM} E_k(x,y;\vth)^s \e(-N\vth) \dd\vth = x^{-k} \Psi(x,y)^s   \bigg( \beta_{\infty} \prod_p \beta_p +  O_s(u_y^{-1}) \bigg). \]
\end{lemma}

\begin{lemma}[Minor arcs for Waring's problem]\label{lem:waring-minor}
Let the notations and assumptions be as in the statement of Theorem~\ref{thm:waringk}, and let~$\Fm$ be defined as above. Then
\[ \int_{\Fm} |E_k(x,y;\vth)|^s \dd\vth \ll_s  x^{-k} \Psi(x,y)^s Q^{-c}   \]
for some~$c = c(k) >0$.
\end{lemma}

Indeed, to deduce Theorem~\ref{thm:waringk} from Lemmas~\ref{lem:waring-major} and~\ref{lem:waring-minor}, it suffices to take~$Q = (\log x)^A$ for some large enough~$A$ so that~$Q^{-c} \ll u_y^{-1}$. In the remainder of this section, we prove the two lemmas.

\subsection{Major arc analysis}

We start by proving Lemma~\ref{lem:waring-major}. For~$\vth\in\major(q,a)$ for some~$0 \leq a \leq q\leq Q$ and~$(a,q)=1$, write~$\vth=a/q+\delta$ with~$|\delta| \leq Qx^{-k}q^{-1}$. Then~$\CQ = q(1+|\delta x^k|) \leq Q$. By Theorem~\ref{thm:estim-majo} we have
\[ \frac{E_k(x,y;\vth)}{\Psi(x,y)} =  \cPhi(\delta x^k,\alpha) H_{a/q}(\alpha) + O \left( \CQ^{-1/k + 2(1-\alpha) + \varepsilon} u_y^{-1} \right) \]
for any~$\ee>0$.  Since
\[ \cPhi(\delta x^k,\alpha) H_{a/q}(\alpha) \ll  \CQ^{-\alpha/k + \varepsilon} \ll \CQ^{-1/k + 1-\alpha + \varepsilon} \]
by Lemmas~\ref{lem:phi-check} and ~\ref{lem:Haq-bound}, we have
\begin{align*} 
\int_{\major(q,a)} \bigg( \frac{E_k(x,y;\vth)}{\Psi(x,y)} \bigg)^s \e(-N\vth) \dd\vth &= H_{a/q}(\alpha)^s \e(-aN/q)\int_{|\delta| \leq Qx^{-k}q^{-1}} \cPhi(\delta x^k,\alpha)^s \e(-N\delta) \dd\delta \\ &+ O\bigg( u_y^{-1}  \int_{|\delta| \leq Qx^{-k}q^{-1}} \CQ^{-s/k + 2s(1-\alpha) + \varepsilon} \dd\delta \bigg). 
\end{align*}
For~$s \geq s_0(k)$, the exponent~$t = s/k - 2s(1-\alpha) - \varepsilon$ satisfies~$t > 2$, and thus the integral in the error term above is bounded by
\[ q^{-t} \int_{|\delta| \leq Qx^{-k}q^{-1}} (1+|\delta x^k|)^{-t} \dd\delta \ll q^{-t} x^{-k}. \]
Moreover, we may extend the integral in the main term above to all of~$\delta \in \R$ with an error~$O(x^{-k} (Q/q)^{1-s\alpha/k})$ (see Lemma~\ref{lem:truncate-integral} below), so that
\[ \int_{\major(q,a)} \bigg( \frac{E_k(x,y;\vth)}{\Psi(x,y)} \bigg)^s \e(-N\vth) \dd\vth = x^{-k} \left( \beta_{\infty} H_{a/q}(\alpha)^s \e(-aN/q) + O(q^{-1+\ee} Q^{1-s\alpha/k} + u_y^{-1}q^{-t}) \right). \]
Summing over all~$0 \leq a \leq q \leq Q$ with~$(a,q)=1$, we obtain
\[
\int_{\FM} \bigg( \frac{E_k(x,y;\vth)}{\Psi(x,y)} \bigg)^s \e(-N\vth) \dd\vth = x^{-k} \left( \beta_{\infty} \sum_{q \leq Q} \sum_{(a,q) = 1} H_{a/q}(\alpha)^s \e(-aN/q) + O(Q^{2-s\alpha/k+\ee} + u_y^{-1}) \right)
\]
since~$\sum q^{-t+1} = O(1)$. The restriction~$q \leq Q$ in the sum above can be removed with an error~$O(Q^{2-s\alpha/k+\ee})$ (see Lemma~\ref{lem:truncate-series} below). Finally, for~$s \geq s_0(k)$, the exponent~$2-s\alpha/k$ is negative and bounded away from~$0$, and thus the error~$O(Q^{2-s\alpha/k+\ee})$ can be absorbed into~$O(u_y^{-1})$ if~$Q = (\log x)^A$ with~$A$ large enough. This completes major arc analysis.

\begin{lemma}[Truncated singular integral]\label{lem:truncate-integral}
Let the notations and assumptions be as above. For any~$\Delta \geq 1$, we have
\[ \int_{|\delta|\leq \Delta x^{-k}} \cPhi(\delta x^k,\alpha)^s \e(-N\delta) \dd\delta  = x^{-k} \big(\beta_{\infty} + O(\Delta^{1-s\alpha/k}) \big). \]
\end{lemma}

\begin{proof}
After a change of variable, the left side above becomes
\[ x^{-k} \int_{|\delta| \leq \Delta}  \cPhi(\delta,\alpha)^s e(-\delta) \dd\delta. \]
The conclusion of the lemma follows from the definition of~$\beta_{\infty}$ in~\eqref{eq:arch-factor} and the estimate
\[ \int_{|\delta| \geq \Delta}  \left| \cPhi(\delta,\alpha) \right|^s \dd\delta \ll \int_{|\delta| \geq \Delta} \delta^{-s\alpha/k} \dd\delta \ll \Delta^{1-s\alpha/k}. \]
\end{proof}

\begin{lemma}[Truncated singular series]\label{lem:truncate-series}
Let the notations and assumptions be as above. For any~$Q \geq 1$, we have
\[ \sum_{q\leq Q} \sum_{(a,q)=1} H_{a/q}(\alpha)^s \e( -aN/q ) = \prod_p \beta_p + O\left( Q^{2 - s\alpha/k + \ee} \right). \]
\end{lemma}

\begin{proof}
In the appendix  we will show that
\[  \sum_{q=1}^{+\infty} \sum_{(a,q)=1}  H_{a/q}(\alpha)^s e( -aN/q ) = \prod_p \beta_p(\alpha). \]
The conclusion of the lemma then follows from
\[ \sum_{q> Q} \sum_{(a,q)=1} | H_{a/q}(\alpha)|^s \ll \sum_{q > Q} q^{1-s\alpha/k+\ee} \ll Q^{2-s\alpha/k+\ee}. \]
\end{proof}

\subsection{Minor arc analysis}

Now we prove  Lemma~\ref{lem:waring-minor}, bounding the minor arc integral by
\[ \sup_{\vth \in \Fm} |E_k(x,y; \vth)|^{0.1} \cdot \int_0^1 |E_k(x,y; \vth)|^{s-0.1} \dd\vth. \]
For~$s \geq s_0(k)$, the exponent~$s - 0.1$ exceeds the threshold~$p_0(k)$ in Theorem~\ref{cor:mv-k}, so that the integral above can be bounded by~$O(\Psi(x,y)^{s-0.1}x^{-k})$.  On the other hand, the minor arc estimate (Proposition~\ref{prop:estim-minor}) implies that
\[ \sup_{\vth\in\Fm} |E_k(x,y;\vth)| \ll \Psi(x,y) Q^{-c}. \]
This completes the proof of Lemma~\ref{lem:waring-minor}.

\appendix

\section{The local factors in friable Waring's problem}

The aim of this appendix is to establish Propositions~\ref{prop:beta-p} about local factors, by first connecting~$\beta_p$ with exponential sums weighted by~$\mu_{p^m}$, and then expressing the exponential sum in terms of the classical ones (corresponding to~$y = x$). 

Let the notations and assumptions be as in the statement of Theorem~\ref{thm:waringk}, and recall Definition~\ref{eq:local-factor}. We have defined~$\mu_q$ for~$q = p^m$ a prime power. Now extend~$\mu_q$ multiplicatively to all~$q$ (so that~$\mu_{q_1q_2}(b) = \mu_{q_1}(b)\mu_{q_2}(b)$ for any~$b$, whenever~$(q_1,q_2) = 1$), and note that the value of~$\mu_q(b)$ depends only on~$(b,q)$. For~$0 \leq a \leq q$ and~$(a,q) = 1$, define the exponential sum
\[ S(x,y; q,a) = \sum_{b\mod{q}} \mu_q(b) \e\bigg( \frac{ab^k}{q} \bigg), \]
which should be compared with the exponential sum appearing in the classical Waring's problem:
\[ S(q,a) = \frac{1}{q} \sum_{b\mod{q}}  \e\bigg( \frac{ab^k}{q} \bigg). \]
Recall the definition of~$H_{a/q}(\alpha)$ in~\eqref{eq:def-Hq}.

\begin{lemma}\label{lem:S=H(a/q)} 
For any~$0 \leq a \leq q$ and~$(a,q) = 1$, we have~$S(x,y;a,q) = H_{a/q}(\alpha)$. 
\end{lemma}

\begin{proof}
By definitions, it suffices to show that for any~$b\mod{q}$ with~$(b,q) = d_1$ we have
\[ \mu_q(b) = \sum_{\substack{d_1d_2 \mid q \\ P(d_1d_2) \leq y}} \frac{\mu(d_2)}{(d_1d_2)^{\alpha} \varphi(q/d_1)}. \]
As functions of~$q$, both sides above are multiplicative in~$q$, so that it suffices to verify this for~$q = p^m$ a prime power. This is a straightforward comparison with the definition of~$\mu_{p^m}(b)$.
\end{proof}

The following lemma says that the probability measure~$\mu_{p^m}$ behaves well under the natural projection~$\Z/p^m\Z \rightarrow \Z/p^{m-\ell}\Z$.

\begin{lemma}\label{lem:mu-translate}
For any prime~$p$, any integers~$0\leq\ell\leq m$, and any~$b\in\Z$, we have the identity
\[ \sum_{u\in\Z/p^{\ell}\Z} \mu_{p^m}(up^{m-\ell}+b) = \mu_{p^{m-\ell}}(b). \]
\end{lemma}

\begin{proof}
First assume that~$(b,p^m) < p^{m-\ell}$. Then~$(up^{m-\ell}+b, p^m)= (b, p^m)$ for each~$u$, and thus the sum is equal to~$p^{\ell}\mu_{p^m}(b)$. This is easily seen to be equal to~$\mu_{p^{m-\ell}}(b)$ from the definition.

Now assume that~$(b, p^m) \geq p^{m-\ell}$. Then the sum becomes
\[ S = \sum_{u\in\Z/p^{\ell}\Z} \mu_{p^m}(up^{m-\ell}) = \sum_{v=m-\ell}^m \varphi(p^{m-v}) \mu_{p^m}(p^v),   \]
where~$\varphi(p^{m-v})$ is the number of~$b\in\Z/p^m\Z$ with~$(b, p^m) = p^v$. If~$p>y$, then the only nonzero term in the sum above appears when~$\ell = m$, and thus~$S = \1_{\ell = m} = \mu_{p^{m-\ell}}(0)$ as desired. If~$p\leq y$, then
\[ S = \sum_{v=m-\ell}^{m-1} \varphi(p^{m-v}) \varphi(p^m)^{-1} p^{(1-\alpha)v} (1-p^{-\alpha}) + p^{-\alpha m} = p^{-\alpha(m-\ell)} = \mu_{p^{m-\ell}}(0),  \]
as desired. This completes the proof. 
\end{proof}

For any positive integer~$q$, define
\[ S(q) = \sum_{a\mod{q}^{\times}} S(x,y;q,a)^s  \e\bigg(\frac{-aN}{q}\bigg) = \sum_{a\mod{q}^{\times}} H_{a/q}(\alpha)^s  \e\bigg(\frac{-aN}{q}\bigg). \]
From the standard fact that
\[ S(x,y;q,a) S(x,y;q',a') = S(x,y; qq', aq'+a'q) \]
for~$(q,q')=(a,q)=(a',q')=1$, it follows that~$S(q)$ is multiplicative in~$q$. 

\begin{lemma}\label{lem:Sq-Mq}
For any positive integer~$q$, let~$M(q)$ be the number of solutions to~$n_1^k+\cdots+n_s^k\equiv N\mod q$ counted with weights given by~$\mu_q$:
\[ M(q) = \ssum{n_1,\cdots,n_s\in\Z/q\Z \\ n_1^k+\cdots+n_s^k\equiv N\mod q} \mu_q(n_1)\cdots\mu_q(n_s). \]
Then
\[ \sum_{d\mid q} S(d) = q M(q). \]
\end{lemma}



\begin{proof} 
Since both sides are multiplicative in~$q$, it suffices to prove the assertion when~$q=p^m$ is a prime power.
By orthogonality, we can write
\[ M(q) = \frac{1}{q} \sum_{a=1}^q \bigg( \sum_{b=1}^q \mu_q(b) \e(ab^k/q) \bigg)^s \e(-aN/q). \]
For any~$d\mid q$, the contribution from those terms with~$(a,q)=d$ is
\[ M_d(q) = \frac{1}{q} \ssum{1\leq a\leq q/d\\ (a,q/d)=1} \bigg( \sum_{b=1}^q \mu_q(b) \e(adb^k/q) \bigg)^s \e(-adN/q). \]
Suppose that~$d=p^{\ell}$ for some~$0\leq \ell \leq m$. If we write~$b=up^{m-\ell}+v$ for some~$1\leq v\leq p^{m-\ell}$ and~$0\leq u<p^{\ell}$, the inner sum over~$b$ becomes
\[ \sum_{u=0}^{p^{\ell}-1} \sum_{v=1}^{p^{m-\ell}} \mu_{p^m}(up^{m-\ell}+v) \e(av^k/p^{m-\ell}) = \sum_{b=1}^{p^{m-\ell}} \bigg(\sum_{u=0}^{p^{\ell}-1} \mu_{p^m}(up^{m-\ell}+b)\bigg) \e(ab^k/p^{m-\ell}) = S(x,y;p^{m-\ell},a)  \]
by Lemma~\ref{lem:mu-translate}. It follows that
\[ M_{p^{\ell}}(p^m) = \frac{1}{p^m} \sum_{\substack{1\leq a \leq p^{m-\ell}\\ (a,p)=1}}S(x,y;p^{m-\ell},a)^s \e(-aN/p^{m-\ell}) = \frac{1}{p^m} S(p^{m-\ell}). \]
This completes the proof.
\end{proof}

The following lemma provides an upper bound for the exponential sum~$S(x,y;q,a)$ by expressing it in terms of the classical sum~$S(q,a)$ (alternatively, one may also proceed directly with the definition~\eqref{eq:def-Hq}).

\begin{lemma}\label{lem:Sqa}
For any~$0 \leq a \leq q$ with~$(a,q) = 1$, we have 
\[ |S(x,y;q,a)| \leq C^{\omega(q)} q^{-\alpha/k}, \]
where~$C \geq 1$ is an absolute constant. In particular,
\[ |S(q)| \ll q^{1-s\alpha/k + \ee} \]
for any~$\ee > 0$.
\end{lemma}

\begin{proof}
By multiplicativity it suffices to prove these when~$q = p^m$ is a prime power. By definition we may express~$S(x,y;q,a)$ in terms of the classical~$S(q,a)$ as follows.
If~$p>y$, then
\[ S(x,y;p^m,a) = \begin{cases} \frac{1}{p-1} ( pS(p^m,a) - 1 ) & \text{if }m\leq k\\ \frac{1}{p-1} (pS(p^m,a) - S(p^{m-k},a)) & \text{if }m>k. \end{cases} \]
If~$p\leq y$, then
\begin{align*}
 S(x,y;p^m,a) = & \frac{(1-p^{-\alpha})(1-p^{\alpha-1})}{1-p^{-1}} \sum_{1\leq v<v_0} p^{-v\alpha} S(p^{m-vk},a) + \frac{1-p^{-\alpha}}{1-p^{-1}} S(p^m,a) \\
 & + \frac{1}{\varphi(p^m)} \left[ p^{m-\alpha v_0}(1-p^{\alpha-1}) + (p^{(1-\alpha)(m-1)}-1)(1-p^{-\alpha}) \right], 
 \end{align*}
where~$v_0 = \lceil m/k \rceil$. Note that 
\[  \frac{1}{\varphi(p^m)} \left[ p^{m-\alpha v_0}(1-p^{\alpha-1}) + (p^{(1-\alpha)(m-1)}-1)(1-p^{-\alpha}) \right] \ll p^{-\alpha v_0} + p^{-\alpha m-1+\alpha} \ll p^{-\alpha m/k}. \]
The claimed bound on~$S(x,y;q,a)$ follows from these using the classical estimate~$|S(q,a)| \ll q^{-1/k}$ (see~\cite[Theorem 4.2]{Vaughan97}) after some straightforward algebra, and the claimed bound on~$S(q)$ follows by the triangle inequality.
\end{proof}

\begin{proof}[Proof of Proposition~\ref{prop:beta-p}]
We start with justifying the existence of the limit in the definition of~$\beta_p$. By Lemma~\ref{lem:Sq-Mq}, we have
\begin{equation}\label{eq:beta-p-M} 
\beta_p = \lim_{m \rightarrow \infty} p^m M(p^m) = \sum_{\ell = 0}^{\infty} S(p^{\ell}).
\end{equation}
By Lemma~\ref{lem:Sqa}, the infinite sum above is absolutely convergent, and more precisely we have
\[ | \beta_p - 1| \ll \sum_{\ell \geq 1} p^{\ell(1-s\alpha/k + \ee)} \ll p^{1-s\alpha/k + \ee}. \]
Hence the infinite product~$\prod_p \beta_p$ converges for~$s \geq s_0(k)$.

It remains to show that~$\beta_p > 0$ for each prime~$p$. For~$p > y$, this follows from the bound on~$|\beta_p-1|$ above. For~$p \leq y$, from the definition of~$\mu_{p^m}(b)$ we have
\[ \mu_{p^m}(b) \geq p^{-m} \cdot \frac{1-p^{-\alpha}}{1-p^{-1}} \]
for any~$b$. This shows that~$\beta_p$ is at least 
\[ \left(\frac{1-p^{-\alpha}}{1-p^{-1}}\right)^s  \]
times the value of~$\beta_p$ in the classical case~$y=x$, which is positive when~$s \geq s_0(k)$ (see Lemma 2.12, 2.13, and 2.15 in~\cite{Vaughan97}).
\end{proof}

Observe that by~\eqref{eq:beta-p-M} and the multiplicativity of~$S(q)$, we have
\[ \prod_p \beta_p = \sum_{q=1}^{+\infty} S(q) = \sum_{q=1}^{+\infty} \sum_{a\mod{q}^{\times}} H_{a/q}(\alpha)^s \e(-aN/q). \]
This was used in proving Lemma~\ref{lem:truncate-series} in the major arc analysis.

\bibliographystyle{plain}
\bibliography{biblio}

\end{document}